\documentclass[11pt,a4paper]{article}
\usepackage{authblk}
\usepackage{geometry}
\usepackage{mathrsfs}
\usepackage{amssymb,amsfonts,amsmath,amsthm}
\usepackage{cleveref}
\newtheorem{theorem}{Theorem}[section]
\newtheorem{lemma}{Lemma}[section]
\newtheorem{corollary}{Corollary}[section]
\newtheorem{proposition}{Proposition}[section]
\newcommand{\diam}[1]{\mathrm{diam}(#1)}

\newcommand{\ext}{C(D,s)}

\newcommand{\prnt}[1]{\left( #1 \right)}

\newcommand{\norm}[1]{\left\|#1\right\|}

\newcommand{\normL}[2]{\norm{#1}_{L^2\prnt{#2}}}
\newcommand{\normHp}[3]{\norm{#1}_{H^{#2}\prnt{#3}}}

\newcommand{\normI}[2]{\norm{#1}_{C\prnt{#2}}}
\newcommand{\normLH}[3]{\norm{#1}_{L^2(\Omega,H^{#2}\prnt{#3})}}

\newcommand{\lebesgue}{\mathcal{L}^{(d)}}
\newcommand{\mc}[1]{\mathcal{#1}}

\newtheorem{assumption}{Assumption}[section]

\usepackage{soul}
\setstcolor{red}

\newcommand{\om}{\omega}
\newcommand{\nsamples}{M}
\newcommand{\nsamplesrun}{m}

\newcommand{\kllevel}{L}
\newcommand{\kllevelrun}{\ell}
\newcommand{\Vh}{\mathcal{V}_h}
\newcommand{\dimVh}{Q_h}
\newcommand{\dimVhrun}{k}
\newcommand{\spn}{\text{span}}
\renewcommand{\vec}[1]{\boldsymbol{#1}}

\newcommand{\lnrf}{\kappa}
\newcommand{\lnrfrv}{K}

\newcommand{\prob}{\mathbb{P}}
\newcommand{\dom}{\mathcal{D}}
\newcommand{\R}{\mathbb{R}}
\newcommand{\N}{\mathbb{N}}
\newcommand{\taperingint}{\tau}
\newcommand{\taperingweight}{w}
\newcommand{\transpose}{T}
\newcommand{\CholStiffnessL}{\vec{L}^{(h)}}
\newcommand{\stiffnessmatrix}{\vec{S}^{(h)}}
\newcommand{\stiffnessmatrixM}{\vec{S}^{(h;\nsamples)}}
\newcommand{\tildestiffnessmatrixM}{\widetilde{\vec{S}}^{(h;\nsamples)}}

\newcommand{\tildstiffnessmatrix}{\widetilde{\vec{S}}^{(h)}}
\newcommand{\massmatrix}{\vec{M}^{(h)}}

\newcommand{\sigmaalg}{\mathcal{F}}

\newcommand{\constC}{C_{\infty}} 	
\newcommand{\constHs}{C_{H^{s}}}

\newcommand{\expec}{\mathbb{E}}
\newcommand{\normc}[2]{\left\|#1 \right\|_{C(#2)}}

\newcommand{\expect}[1]{\mathbb{E}\left[#1 \right] }
\newcommand{\decayclass}{\mathcal{F}}
\title{On the Numerical Approximation of the Karhunen-Lo\`{e}ve Expansion for Random Fields with Random Discrete Data
}
\author[1]{Michael Griebel}
\author[2]{Guanglian Li}
\author[3]{Christian Rieger}
\affil[1]{Institut f\"ur Numerische Simulation,
            Universit\"at Bonn,
            Friedrich-Hirzebruch-Allee 7, 53115 Bonn, Germany and Fraunhofer SCAI, Schloss Birlinghoven, 53754 Sankt Augustin, Germany. (griebel@ins.uni-bonn.de)}
\affil[2]{Department of Mathematics, The University of Hong Kong, Pokfulam Road, Hong Kong. (lotusli@maths.hku.hk)}
\affil[3]{Philipps-Universit\"at Marburg, Fachbereich Mathematik und Informatik, AG Numerik,
Hans-Meerwein-Stra\ss{}e 6, 35032 Marburg. (riegerc@mathematik.uni-marburg.de)}
\date{\today}
\begin{document}
\maketitle
\begin{abstract}
In many applications, random fields reflect uncertain parameters, and often their moments are part of the modeling process and thus well known. However, there are practical situations where this is simply not the case. Therefore, we do not assume that we know
moments or expansion terms of the random fields, but only have discretized samples of them. The main contribution of this paper concerns the approximation of the true covariance operator from these finite measurements. We derive explicit error estimates that include the finite-rank approximation error of the covariance operator, the Monte Carlo-type error for sampling in the stochastic domain, and the numerical discretization error in the physical domain. For this purpose, we use modern tapering covariance estimators adapted to high-dimensional applications, where the dimension is introduced by the resolution of the measurement process. This allows us to give sufficient conditions on the three discretization parameters to guarantee that the error is kept below a prescribed accuracy $\varepsilon$. 
\end{abstract}

\begin{center}
\end{center}

\begin{center}
MSCcodes:
41A25, 41A35, 60F10, 65D40
\end{center}

\section{Introduction}\label{sec:intro}
Mathematical models with random coefficients or random input data have been widely used to describe applications
that are affected by some degree of uncertainty due to incomplete or insufficient information subject to some degree of uncertainty about the problem. The range of applications is broad and diverse, including uncertainty quantification with ordinary differential equations, partial differential equations, and integro-differential equations, or problems in machine learning and data analysis in oil field modeling, quantum mechanics, or finance.

Often, the random data is modeled as a Gaussian random field over a probability space $(\Omega,\sigmaalg,\prob)$. Thus, the random field is 
described by its mean and its covariance. By subtracting its mean, we can further assume that the random field is centered. 
For numerical simulations, realizations of this random field must then be derived.

In contrast to this common approach, we intend to estimate the statistical information from experimental data. For this purpose, we  assume that there is an unknown centered Gaussian field $\lnrf$, which is
consistent with the covariance kernel 
\begin{align}\label{carlemannkern}
\text{Cov}_{\lnrf}(\vec{x},\vec{x}^{\prime})
&=:R(\vec{x},\vec{x}^{\prime})=\int_{\Omega}\lnrf (\vec{\om},\vec{x}) \lnrf (\vec{\om},\vec{x}^{\prime})\, \mathrm{d}\prob(\vec{\om}) 
=\sum_{\kllevelrun \in \N} \lambda_{\ell}\sum_{k=1}^{N(\kllevelrun,k)} \phi_{\kllevelrun,k}(\vec{x}) \phi_{\kllevelrun,k}(\vec{x^{\prime}})
\nonumber \\&= \sum_{\kllevelrun=1}^{\infty} \lambda_{\kllevelrun} \phi_{\kllevelrun} (\vec{x})\phi_{\kllevelrun}(\vec{x}^{\prime}) \approx \sum_{\kllevelrun=1}^{L} \lambda_{\kllevelrun} \phi_{\kllevelrun} (\vec{x})\phi_{\kllevelrun}(\vec{x}^{\prime})=: R^{\kllevel}(\vec{x},\vec{x}^{\prime})
\end{align}
where we have the eigenvalues sorted in strictly decreasing order, i.e. $\lambda_1 > \lambda_2 \dots \ge 0$ for the first series expansion, and we understand the last (notationally simpler) expression as repeating the eigenvalues according to their multiplicity. This will be used from now on for simplicity of notation.
The last sum is usually called the finite noise approximation. This approximation, and in particular the truncation parameter $L$, is justified by a fast decay of the eigenvalues in the Mercer-type expansion. 

For our analysis we assume spatial regularity of the random field $\lnrf$. We have:
\begin{assumption}\label{A:11} Let $\lnrf:\Omega \times \dom \to \R$ and thus $\lnrfrv:\Omega \to \R^{\dom}$ be  contained in the Bochner space
$L^{\infty}(\Omega, H^{s}(\dom))$ for some  $s\ge 0$, where for $s=0$ we have $H^{0}(\dom)=L^{2}(\dom)$.
Alternatively, let $\lnrf:\Omega \times \dom \to \R$ and hence $\lnrfrv:\Omega \to \R^{\dom}$ be contained in the Bochner space
$L^{\infty}(\Omega, H^{s}(\dom))$ for some $s>\frac{d}{2}$.
\end{assumption}
We refer to \cite[Theorem 5.2]{steinwart2017} for conditions on $\lnrf$ that allow for such types of bounds.

We also assume some moment bounds:
\begin{assumption}\label{A:00}
Let $\lnrf:\Omega \times \dom \to \R$, we assume
$\constC :=\left(\int_{\Omega} \normc{\lnrf\left(\vec{\om} , \cdot\right)}{\dom}^2  \, \mathrm{d}\prob(\vec{\om})\right)^{\frac{1}{2}}<\infty$.
Sometimes we need the stronger assumption
$\constHs :=\left(\int_{\Omega} \normHp{\lnrf\left(\vec{\om} , \cdot\right)}{s}{\dom}^2  \, \mathrm{d}\prob(\vec{\om})\right)^{\frac{1}{2}}<\infty$ for later use.
\end{assumption}
These estimates imply that
$\left| \expect{\lnrf(\cdot,\vec{x})}\right| \le \constC$ and $\text{Cov}(\vec{x},\vec{x}^{\prime})\le 4\constC^2$ by H\"older's inequality.
Since our analysis is aimed at practical applications, we consider mostly discretized function spaces. For the remainder of this article, we will focus on finite element spaces of the form
\begin{align}\label{eq:FEspace}
\Vh:=\{v\in H^{1}(\dom):v|_{K}\in P^{
\lceil s \rceil }(K) \text{ for all } K\in \mc{T}_h\} \subset \R^{\dom},
\end{align}
where $\mathcal{T}_{h}$ is a regular quasi-uniform triangulation over the
physical domain $\dom$ with a maximal mesh size $h$ and $ \lceil s \rceil  \in \N$ depends on the spatial regularity of $\lnrf$.

We denote the dimension of $\Vh$ by $\dimVh <\infty$. Furthermore, we expect $\dimVh = \mathcal{O}(h^{-d}s^d)$.
For the rest of the paper we will fix some arbitrary basis $\Vh=\spn\left\{\theta_{\dimVhrun}^{(h)} \ : \  1 \le \dimVhrun \le \dimVh \right\}$.
Let $\vec{\theta}^{(h)}$ be a generic basis giving rise to the following vector field
$\vec{\theta}^{(h)}: \dom \to   \R^{\dimVh}$, with $ \vec{x} \mapsto \left(\theta^{(h)}_{1}(\vec{x}),\dots,\theta^{(h)}_{\dimVh}(\vec{x}) \right)^{\transpose}$.
To fix the notation: any function $v^{(h)} \in \Vh$ can be identified with a vector $\vec{V}^{(h)} = \left( V^{(h)}_{1},\dots, V^{(h)}_{\dimVh} \right)^{\transpose}\in \R^{\dimVh}$ via
$v^{(h)}= \vec{V}^{(h)}\cdot \vec{\theta}^{(h)}$.

Next, we consider the $L^2$ projection map $\Pi_{\Vh}:L^{2}(\dom) \to \Vh$ and note that its approximation rate is given in \cite[Theorem 4.4.20]{MR2373954} as
\begin{align}
\normL{v-\Pi_{\Vh}v}{\dom}&\leq C_{\Pi_{\Vh;L^2}} h^{s}\normHp{v}{s}{\dom} \text{ for all } v\in H^s(\dom),\label{eq:approxL2}\\
\normI{v-\Pi_{\Vh}v}{\dom}&\leq C_{\Pi_{\Vh;C}}h^{s-\frac{d}{2}}\normHp{v}{s}{\dom} \text{ for all } v\in H^s(\dom) \text{ for } s>d/2. \label{eq:approxLinfty}
\end{align}
Here, the positive constants $C_{\Pi_{\Vh;L^2}}$ and $C_{\Pi_{\Vh;C}}$ depend only on the shape regularity parameter of $\mathcal{T}_h$ and are independent of the mesh size $h$. 
Since we mostly deal with discretized functions, we consider 
\begin{equation*}
\lnrf^{(h)}\left(\vec{\om} , \cdot\right):=\Pi_{\Vh} \lnrf(\vec{\om},\cdot) = \vec{\lnrfrv}^{(h)}(\vec{\omega}) \cdot \vec{\theta}^{(h)}.
\end{equation*}
As in the continuous case, this implies
 $\expect{\lnrf^{(h)}(\cdot, \vec{x})}\le\constC^{(h)}$ and $\text{Cov}^{(h)}(\vec{x},\vec{x}^{\prime})\le 4(\constC^{(h)})^2$ for all $\vec{x},\vec{x}^{\prime} \in \dom$, where $\constC^{(h)} :=\left(\int_{\Omega} \normc{\lnrf^{(h)}\left(\vec{\om} , \cdot\right)}{\dom}^2  \, \mathrm{d}\prob(\vec{\om})\right)^{\frac{1}{2}}<\infty$.
Analogously, 
$C^{(h)}_{H^s} :=\left(\int_{\Omega} \norm{\lnrf^{(h)}\left(\vec{\om} , \cdot\right)}_{H^s(\dom)}^2  \, \mathrm{d}\prob(\vec{\om})\right)^{\frac{1}{2}}<\infty$.
Furthermore, the only theoretically available covariance is given in the centered case as
\begin{align}
\label{truecovariance}
\vec{\Sigma}_{\vec{\lnrfrv}^{(h)}}
=\expect{\vec{\lnrfrv}^{(h)} \otimes \vec{\lnrfrv}^{(h)} } \in \R^{\dimVh \times \dimVh}.
\end{align}
Note that this matrix is large and its dimension depends on the spatial discretization parameter.

Next, we assume that we have access to $\nsamples$ independent identically distributed discrete samples  $\vec{\lnrfrv}^{(1;h)}:=\vec{\lnrfrv}^{(h)}(\vec{\om}_1), \dots, \vec{\lnrfrv}^{(\nsamples;h)}:=\vec{\lnrfrv}^{(h)}(\vec{\om}_M) $. The eigensystem of this matrix gives rise to a computable Mercer-like expansion
\begin{align}\label{mercerh3}
R^{(h;M)}(\vec{x},\vec{x}^{\prime})& :=\sum_{\kllevelrun=1}^{\dimVh}\lambda^{(h;M)}_{\kllevelrun}
\phi_{\kllevelrun}^{(h;M)}(\vec{x})\phi_{\kllevelrun}^{(h;M)}(\vec{x}^{\prime}) \approx R^{(\kllevel;h;M)}(\vec{x},\vec{x}^{\prime})\nonumber \\&:= \sum_{\kllevelrun=1}^{\kllevel}\lambda^{(h;M)}_{\kllevelrun}
\phi_{\kllevelrun}^{(h;M)}(\vec{x})\phi_{\kllevelrun}^{(h;M)}(\vec{x}^{\prime}),
\end{align}
where the  $(h;M)$--notation indicates that we are using an estimator based on the $\nsamples$  finite samples $\vec{\lnrfrv}^{(1;h)}, \dots, \vec{\lnrfrv}^{(\nsamples;h)}$. In addition, $\kllevel$ indicates the finite-noise approximation. 

Since $\dimVh \sim h^{-d}$, the summation in \eqref{mercerh3} is subject to the well-known curse of dimensionality. Consequently, the estimation of the covariance matrix becomes more difficult the smaller the values of $h$ become. Therefore, we assume that the covariance matrix $\vec{\Sigma}_{\vec{\lnrfrv}^{(h)}} \in \R^{\dimVh \times \dimVh}$ has a certain off-diagonal decay, i.e. $\vec{\Sigma}_{\vec{\lnrfrv}^{(h)}} \in \decayclass_{\alpha}=\decayclass_{\alpha}(C_{\decayclass;1},C_{\decayclass;2}),$ where the class is defined later in \eqref{taperingclass}.
Here, $\alpha$ modulates the decay rate and $ C_{\decayclass;1},C_{\decayclass;2}$ are positive constants.
This particular class of matrices contains several relevant examples of discrete covariance functions, see \cite{bickel2008} and the survey \cite{cai2016}.
In addition, the choice of this class of matrices also affects the approximation of continuous covariance functions by so-called tapering weights, see \eqref{taperingweights}. We will use the optimal rates derived in \cite{bickel2008}.

The main goal of this article is to derive a bound on the expected approximation error \\$\expect{\|R-R^{(\kllevel;h;M)}\|_{L^{2}(\dom \times \dom)}}$.
To do this, we couple the discretization parameters $\kllevel, h, \nsamples $ and consider their limits $h\to 0,\kllevel \to \dimVh \to \infty, \nsamples  \to \infty$ to recover the true continuous covariance operator. 
Our 
main result is 
\cref{thm:final2}. It states that 
\begin{align*}
\expect{\norm{{R}-R^{(\kllevel;h;\nsamples)}}_{L^{2}(\dom\times\dom)}} &\lesssim  \kllevel^{-\frac{2s}{d}-\frac{1}{2}}  + \left( \kllevel^{\frac{1}{2}}+ G(\kllevel) \right) \tilde{\rho}^{\frac{1}{2}}_{h}(\nsamples) \lambda_{\max}\left(\massmatrix\right) \nonumber\\
	&+  L^{\frac{1}{2}} h^{-d} \exp\left(- \nsamples  \rho_1 H(\kllevel)\lambda_{\max}^{-2} \left(\massmatrix\right)  \right),
\end{align*}
where $G(\kllevel)$ and $H(\kllevel)$ are functions of the truncation parameter $\kllevel$, which depend on the spectral properties \eqref{HL} of the true operator,  $\tilde{\rho}_{h}(\nsamples)$ measures the approximation quality of our (tapering) estimator, $\massmatrix$ is a classical finite element mass matrix, and $\rho_1$ denotes a constant given in the proof of \cref{prop:spectralgap} which is determined by the sub-Gaussian property of the  random variables involved.
This allows for sufficient conditions on the three discretization parameters to guarantee that an error below a prescribed accuracy $\varepsilon$ is achieved.

The remainder of this paper is organized as follows: In 
\cref{sec:prelim}, we introduce the notation and review some basic facts. In \cref{sec:contkl}, we recall the KL-expansion and sharp eigenvalue estimates. In \cref{sec:apprcov}, we present the reconstruction of the covariance matrix as a stiffness matrix with respect to the finite element discretization. We focus on the discrete spatial approximation in \cref{sec:spatial} and the statistical approximation in \cref{subsec:approxcovhM}. In \cref{sec:tails}, we review some tail estimates for Gaussian random variables.
In \cref{subsec:approxcovhMfinal}, we give bounds on the sampling covariance error. In \cref{sec:covoprec}, we present our final error estimates for the covariance operator reconstruction and give sufficient conditions on the three discretization parameters to ensure that the error is below a prescribed accuracy $ \varepsilon$. We give some concluding remarks in \cref{sec:conclusion}.

\subsection{Notation and basic facts}\label{sec:prelim}
We start with some notation. Let two Banach spaces $V_1$ and $V_2$ be given. Then, $\mathcal{B}(V_1,V_2)$ stands for the Banach space composed of all continuous linear operators from $V_1$ to $V_2$ and $\mc{B}(V_1)$ stands for $\mc{B}(V_1, V_1)$. The set of non-negative integers is denoted by  $\mathbb{N}$. For each index $\alpha\in \mathbb{N}^d$, $|\alpha|$ is the
sum of its components. The letters $\kllevel$, $\nsamples$ and $h$ are reserved for the truncation number of the KL modes, the number of sampling points and the mesh size. We write $A\lesssim B$ if $A\leq cB$ for some absolute constant $c$ which is independent of $\kllevel$, $\nsamples$ and $h$, and we also write $A\gtrsim B$. Furthermore, for any $s\in \mathbb{N}$, $1\leq p\leq \infty$, we follow \cite{Adam78} and 
use the classical Sobolev spaces, see \cite{MR2373954}. 
The space $W_{0}^{s,p}(D)$ is the closure of $C^{\infty}_{0}(\dom)$ in $W^{s,p}(D)$. Its dual space is $W^{-s,q}(\dom)$, with ${1}/{p}+{1}/{q}=1$. We also use $H^{s}(\dom)=W^{s,p}(\dom)$ for $p=2$. Finally, $(\cdot,\cdot)_{\dom}$ denotes the inner product in $L^2(\dom)$.

\section{Karhunen-Lo\`{e}ve expansion: Continuous level}\label{sec:contkl}
This section deals with the Karhunen-Lo\`{e}ve expansion of a centered Gaussian random field $\lnrf$. Let $\lebesgue(\vec{x})$ be the Lebesgue measure on the physical domain $\dom$. For simplicity, $L^{2}(\dom)$ and $L^2(\Omega)$ are short for $L^{2}(\dom;\mathrm{d}\lebesgue(\vec{x}))$ and $L^{2}(\Omega;\mathrm{d}\prob)$. We denote the associated integral operator $\mathcal{S}: L^{2}(\dom)\rightarrow L^{2}(\Omega)$ by
\begin{align}\label{eq:S}
\mathcal{S}: L^{2}(\dom)\rightarrow  L^{2}(\Omega), \quad \mathcal{S}v=\int_{\dom}\lnrf (\vec{\om},\cdot)v\, \mathrm{d}\lebesgue \text{ and } \mathcal{S}^{*}v=\int_{\Omega}\lnrf(\vec{\om},\cdot)v\, \mathrm{d}\prob(\vec{\om}).
\end{align}
where $\mathcal{S}^{*}: L^{2}(\Omega)\rightarrow L^{2}(\dom)$ is the adjoint operator. Let $\mathcal{R}: L^2(\dom)\rightarrow L^2(\dom)$ be defined by  $\mathcal{R}:=\mathcal{S}^{*}\mathcal{S}$.
Then $\mathcal{R}$ is a nonnegative self-adjoint Hilbert-Schmidt operator with kernel
$R\in L^2(\dom \times \dom)$, given in \eqref{carlemannkern}, i.e. 
$R(\vec{x},\vec{x}^{\prime})=\int_{\Omega}\lnrf (\vec{\om},\vec{x}) \lnrf (\vec{\om},\vec{x}^{\prime})\, \mathrm{d}\prob(\vec{\om}) = 
\expec\left[\lnrf (\cdot,\vec{x}) \lnrf (\cdot,\vec{x}^{\prime}) \right]$.
The standard spectral theory for compact operators \cite{yosida78} implies that the operator $\mathcal{R}$ has at most countably many discrete eigenvalues, with zero being the only accumulation point, and each non-zero eigenvalue having only finite multiplicity. Let $\{\lambda_{\kllevelrun}\}_{\kllevelrun=1}^{\infty}$ be the sequence of eigenvalues (with multiplicity counted) associated with
$\mathcal{R}$, which are non-increasingly ordered, and let $\{\phi_{\kllevelrun}\}_{\kllevelrun=1}^\infty$ be the corresponding eigenfunctions which are orthonormal in $L^2(\dom)$.
Furthermore, for any $\lambda_{\kllevelrun}\neq 0$, define
  $\psi_{\kllevelrun}(\vec{\om}):=\frac{1}{\sqrt{\lambda_{\kllevelrun}}}\int_{\dom}\lnrf(\vec{\om},\vec{x})
  \phi_{\kllevelrun} \, \mathrm{d}\lebesgue(\vec{x})$.
One can verify that the sequence $\{\psi_{\kllevelrun}\}_{{\kllevelrun}=1}^\infty$ is uncorrelated and
orthonormal in $L^2(\Omega)$ and therefore $\{\psi_{\kllevelrun}\}_{{\kllevelrun}=1}^{\infty}$ are i.i.d normal random functions.
The KL expansion of the bivariate function $\lnrf$ then refers to the expression
\begin{equation}
\label{eq:KL}
\lnrf(\vec{\om},\vec{x})=\sum\limits_{\kllevelrun=1}^{\infty}\sqrt{\lambda_{\kllevelrun}}\phi_{\kllevelrun}(\vec{x})\psi_{\kllevelrun}(\vec{\om}),
\end{equation}
where the series converges in $L^{2}(\Omega) \otimes L^{2}(\dom )\cong L^2( \Omega\times D)$.

\subsection{$\kllevel$-term truncation in case of continuous Karhunen-Lo\'{e}ve expansion}
Now we will truncate the KL expansion and discuss the resulting error. There are numerous articles on the $\kllevel$-term KL approximation for random fields. For example, in \cite{Schwab:2006:KAR:1167051.1167057} decay rates for the eigenvalues of covariance kernels with certain regularity were considered and generalized fast multipole methods for solving the associated eigenvalue problems were studied. The robust computation of  eigenvalues for smooth covariance kernels was treated in \cite{Todor2006}. A comparison of the $\kllevel$-term KL truncation and the sparse grid approximation was given in \cite{Griebel.Harbrecht:2017}.

The result of this section is based on the paper \cite{griebel2017decay}, which proves a sharp eigenvalue decay rate under a mild assumption on the regularity of the bivariate function $\lnrf $ in the physical domain. 
\cref{A:11} yields the following eigenvalue decay estimate \cite[Theorems 3.2, 3.3, and 3.4]{griebel2017decay} and will be used repeatedly.
\begin{theorem}
\label{thm:truncationError}
Assume that 
\cref{A:11} holds with $s>d/2$. Then, it holds
\begin{align}
\label{thm:truncationErrorLambda}
{{\lambda_{\kllevelrun}}}&\le C_{\ref{thm:truncationError}}\kllevelrun^{-\frac{2s}{d}-1} \quad \text{ for all $\kllevelrun\ge 1$  and }\\
\Big\|{\sum\limits_{\kllevelrun>\kllevel}\sqrt{\lambda_{\kllevelrun}}
\phi_{\kllevelrun}\psi_{\kllevelrun}}\Big\|_{L^2(\Omega\times \dom)}&\leq C_{\ref{thm:truncationError}}^{1/2}\sqrt{\frac{d}{2s}}(\kllevel+1)^{-\frac{s}{d}} \quad \text{ when $\kllevel$ is sufficiently large.}
\label{thm:truncationErrorTrunc}
\end{align}
\end{theorem}
\noindent 
Here the constant is
$C_{\ref{thm:truncationError}}:=\diam{\dom}^{2s}C_{\rm em}(d,s)\ext\normLH{\lnrf}{s}{\dom} ^2$, and $C_{\rm{em}}(d,s)$ denotes an embedding constant between certain Lorentz sequence spaces, and $\ext$ is a constant depending only on $\dom$ and $s$, see \cite{griebel2017decay} for details.

The next lemma gives a regularity result for the eigenfunctions $\{\phi_{\kllevelrun}\}_{\kllevelrun=1}^{\infty}$, see also \cite[Theorem 3.1 \& Remark after Assumption 3.1]{griebel2017decay}:
\begin{lemma}[Regularity of the eigenfunctions $\{\phi_{\kllevelrun}\}_{\kllevelrun=1}^{\infty}$] \label{lem:regeigenfunctions}
Assume that \cref{A:11} holds for $s>d/2$. Then for all $0\leq \beta \leq 1$, there holds
\begin{align}\label{eq:phi_theta}
\normHp{\phi_{\kllevelrun}}{\beta s}{\dom}\leq \lambda^{-\frac{\beta}{2}}_{\ell}.
\end{align}
\end{lemma}
\noindent  
If the eigenvalue bound in \eqref{thm:truncationErrorLambda} are sharp, the bound \ref{eq:phi_theta} reduces to 
\begin{align*}
\normHp{\phi_{\kllevelrun}}{\beta s}{\dom}\leq C(\dom,d, s) \ell^{\beta \left(\frac{s}{d}+\frac{1}{2} \right)} .
\end{align*}
Here $C(\dom,d, s)$ denotes a positive constant that depends only on $\dom$, $d$, and $s$.
The sharpness of \ref{thm:truncationErrorLambda} is, however, only known for the whole function class, see \cite[Theorem 3.4]{griebel2017decay}. 
We also get a bound in the uniform norm by using Sobolev's embedding theorem for $s=\frac{d}{2}+\epsilon>\frac{d}{2}$, i.e. we have
$\normI{\phi_{\kllevelrun}}{\dom} \le \normHp{\phi_{\kllevelrun}}{s}{\dom} \le C(\dom,d, s)\kllevelrun^{\frac{s}{d}+\frac{1}{2}} =
C(\dom,d)\kllevelrun^{1+\epsilon}$ 
with $\kllevelrun$ sufficiently large,
where we use $\epsilon>0$ for an arbitrary small number that can change from line to line.
Moreover, the eigenfunctions $\{\phi_{\kllevelrun}\}_{\kllevelrun=1}^\infty$ are optimal in the sense that the mean square error resulting from a finite-rank approximation of $\lnrf$ is minimized \cite{ghanem2003stochastic}.

\section{Approximation of the covariance matrix}\label{sec:apprcov}
\subsection{Spatial approximation}\label{sec:spatial}
In this section we examine the influence of the spatial discretization.
For now, we assume that we have access to the semi-discrete covariance function
\begin{align}\label{integralkernelh}
R^{(h)}(\vec{x},\vec{x}^{\prime}):=\int_{\Omega}\lnrf^{(h)} (\vec{\om},\vec{x}) \lnrf^{(h)} (\vec{\om},\vec{x}^{\prime})\, \mathrm{d}\prob(\vec{\om})=\expec \left[\lnrf^{(h)} (\cdot,\vec{x}) \lnrf^{(h)} (\cdot,\vec{x}^{\prime}) \right]
\end{align}
and the corresponding integral operator
$\mc{R}^{(h)}f(\vec{x}):=\int_{\dom} f(\vec{x}^{\prime})R^{(h)}(\vec{x},\vec{x}^{\prime}) \, \mathrm{d}\vec{x}^{\prime}$.
Then we have the following result:
\begin{lemma}[Semi-discrete spatial approximation error estimate]
Let $R\in L^2(\dom \times \dom)$ be defined as in \eqref{carlemannkern} and let its numerical approximation $R_h$ be defined as in \eqref{integralkernelh}. Then it holds
\begin{align}\label{integralkernelerror}
\normc{R^{(h)} -R}{\dom \times \dom}
\le C_{\Pi_{\Vh;C}}(\constC+\constC^{(h)}) \constHs h^{s-\frac{d}{2}}.
\end{align}
Furthermore, for the associated integral operators  $\mathcal{R}$ and $\mathcal{R}_h$, it holds that
\begin{align}\label{integraloperatorerror}
\norm{\mc{R}-\mc{R}^{(h)}}_{\mc{B}(L^2(\dom))}\le C_{\Vh;L^2}(\constC\constHs+\constC^{(h)}\constHs^{(h)}) h^{s}.
\end{align}
\end{lemma}
\begin{proof}
We define $n_h(\vec{\om}):=\normc{\lnrf^{(h)} (\vec{\om},\cdot)}{\dom}$ and $n(\vec{\om}):=\normc{\lnrf (\vec{\om},\cdot)}{\dom}$. Then
\begin{align*}
\normc{R^{(h)} -R}{\dom \times \dom} \le  C_{\Pi_{\Vh;L^2}}h^{s-\frac{d}{2}} \int_{\Omega}\normHp{\lnrf (\vec{\om},\cdot)}{s}{\dom} \left(n_h(\vec{\om})+n(\vec{\om})\right)\, \mathrm{d}\prob(\vec{\om}).
\end{align*}
This, along with the \cref{A:00} for $s>d/2$, results in
\begin{align*}
\normc{R^{(h)} -R}{\dom \times \dom}
\le C_{\Pi_{\Vh;C}} \constHs (\constC+\constC^{(h)})  h^{s-\frac{d}{2}} .
\end{align*}
This proves the first statement \eqref{integralkernelerror}.
Now notice that
\begin{align*}
\norm{\mc{R}-\mc{R}^{(h)}}_{\mc{B}(L^2(\dom))}:=\sup_{\genfrac{}{}{0pt}{}{v \in L^2(\dom)}{\normL{v}{\dom}=1}}\normL{\mathcal{R}v-\mathcal{R}^{(h)}v}{\dom}
\le \normL{R-R^{(h)}}{\dom\times\dom}.
\end{align*}
Using \eqref{carlemannkern} and \eqref{integralkernelh}, we get
$
\normL{R-R^{(h)}}{\dom\times\dom}=\normL{\expec\left[\kappa \kappa
-\kappa^{(h)} \kappa^{(h)} \right] }{\dom\times\dom}.
$
By an application of Jensen's inequality, i.e. $f(\expec{X}) \le \expec{f(X)}$ with the convex function $f(X)=\normL{X}{\dom\times\dom}$, we get the bound 
$
\normL{{R}-{R}^{(h)}}{\dom\times\dom}
\le C_{\Vh;L^2}(\constC \constHs+\constC^{(h)} \constHs^{(h)}) h^{s},
$
where we used \eqref{eq:approxL2} in the last step.
This proves the second statement \eqref{integraloperatorerror}.
\end{proof}
Next, we consider the eigenvalue problem for the operator $\mc{R}^{(h)}$, i.e. finding $(\phi_{\kllevelrun}^{(h)},\lambda^{(h)}_{\kllevelrun})\in \Vh \times \R$ such that
\begin{align}\label{genalizedeigenvalue}
\left(\mc{R}^{(h)}\phi_{\kllevelrun}^{(h)},v^{(h)}\right)_{\dom}=\lambda^{(h)}_{\kllevelrun} \left(\phi_{\kllevelrun}^{(h)},v^{(h)}\right)_{\dom} \quad \text{for all  } v^{(h)}\in \Vh,
\end{align}
where $\Vh$ is defined in \eqref{eq:FEspace}.
Using $\phi_{\kllevelrun}^{(h)}:= \vec{\Phi}^{(h)}_{\kllevelrun}\cdot \vec{\theta}^{(h)}=\sum_{\dimVhrun=1}^{\dimVh} \phi_{\kllevelrun;\dimVhrun} \theta_{\dimVhrun}^{(h)}  \in \Vh$,
for fixed but arbitrary $ v^{(h)} =  \vec{V}^{(h)}\cdot \vec{\theta}^{(h)} \in \Vh$, i.e. for fixed but arbitrary $ \vec{V}^{(h)}\in \R^{\dimVh}$, we derive the identity
$
\left(\mc{R}^{(h)}\phi_{\kllevelrun}^{(h)},v^{(h)}\right)_{\dom}=
 \left( \vec{V}^{(h)}\right)^{\transpose} \massmatrix  \vec{\Sigma}_{  \vec{\lnrfrv}^{(h)}}\massmatrix \vec{\Phi}_{\kllevelrun}^{(h)}. 
 $
Here, $\massmatrix$ and $\stiffnessmatrix$ are the mass matrix and the stiffness matrix, respectively, defined by
\begin{align}\label{stiffnessmatrix2}
\massmatrix:=(\vec{\theta}^{(h)},{\vec{\theta}^{(h)}}^{\transpose})_{\dom} \quad \text{and} \quad
\stiffnessmatrix:=\massmatrix\vec{\Sigma}_{\vec{\lnrfrv}^{(h)}}\massmatrix.  
\end{align}
Note that the symmetry of the mass matrix $\massmatrix$ implies the symmetry of $\stiffnessmatrix$. 

Thus, the  corresponding matrix form of the generalized eigenvalue problem \eqref{genalizedeigenvalue} is to search for the eigenpair $(\vec{\Phi}^{(h)}_{\kllevelrun} ,\lambda^{(h)}_{\kllevelrun})\in \mathbb{R}^{\dimVh}\times \mathbb{R}$ for $\kllevelrun=1,\cdots,\dimVh$, satisfying
\begin{align}\label{genalizedeigenvalue2}
\stiffnessmatrix \vec{\Phi}^{(h)}_{\kllevelrun} =\lambda^{(h)}_{\kllevelrun} \massmatrix \vec{\Phi}^{(h)}_{\kllevelrun}.
\end{align}
We transform \eqref{genalizedeigenvalue2} into a conventional eigenvalue problem. Note that the mass matrix $\massmatrix$ is symmetric positive-definite, and thus has a 
Cholesky factorization
$\massmatrix=:\CholStiffnessL\left(\CholStiffnessL\right)^{\transpose} \in \R^{\dimVh \times \dimVh}$.
Now let $\tildstiffnessmatrix$ be defined by
$\tildstiffnessmatrix  :=\left(\CholStiffnessL\right)^{\transpose}\vec{\Sigma}_{\vec{\lnrfrv}^{(h)}}\CholStiffnessL$.
Then we can solve for $(\widetilde{\vec{\Phi}}^{(h)}_{\kllevelrun},\lambda^{(h)}_{\kllevelrun}) \in \R ^{\dimVh} \times \R$ for $\kllevelrun=1,\cdots,\dimVh$, satisfying
\begin{align}\label{genalizedeigenvalue3}
\tildstiffnessmatrix \widetilde{\vec{\Phi}}^{(h)}_{\kllevelrun} 
= \lambda^{(h)}_{\kllevelrun} \widetilde{\vec{\Phi}}^{(h)}_{\kllevelrun},
\end{align}
which is a conventional eigenvalue problem for a symmetric positive-definite matrix.
 By construction we have that the eigenvalue problems \eqref{genalizedeigenvalue2} and \eqref{genalizedeigenvalue3} are  equivalent in the following sense:
\begin{align*}
&(\widetilde{\vec{\Phi}}^{(h)}_{\kllevelrun},\lambda^{(h)}_{\kllevelrun}) \in \R^{\dimVh} \times \R
\text{ is an eigenpair of \eqref{genalizedeigenvalue3}}\Longleftrightarrow\\&
\left( \vec{\Phi}^{(h)}_{\kllevelrun}:=\left(\CholStiffnessL\right)^{-\transpose}\widetilde{\vec{\Phi}}^{(h)}_{\kllevelrun},
\lambda^{(h)}_{\kllevelrun}\right) \in \R^{\dimVh} \times \R
\text{ is an eigenpair of \eqref{genalizedeigenvalue2}.}
\end{align*}
Note that,  with $\phi_{\kllevelrun}^{(h)}= \vec{\Phi}^{(h)}_{\kllevelrun}\cdot \vec{\theta}^{(h)}$, we get a Mercer-like representation 
\begin{align}\label{mercerh}
R^{(h)}(\vec{x},\vec{x}^{\prime})
\approx R^{(\kllevel;h)}(\vec{x},\vec{x}^{\prime}):= \sum_{\dimVhrun=1}^{\kllevel}\lambda^{(h)}_{\dimVhrun}\phi_{\dimVhrun}^{(h)}(\vec{x})\phi_{\dimVhrun}^{(h)}(\vec{x}^{\prime})
\end{align}
for $1\le \kllevel \le \dimVh $.

We now give a classical error analysis by means of the
approximation theory of conforming finite element methods.
Consider the error operator
\begin{align}\label{eq:notationErr}
\mc{E}^{(h)}=\mc{R}-\mc{R}^{(h)}:L^{2}(\dom) \to L^{2}(\dom),
\end{align}
which is a self-adjoint
operator on $L^2(\dom)$. We have the following error representation by Galerkin orthogonality: 
\begin{lemma}\label{lemma:errOper}
The error operator $\mc{E}^{(h)}$ has the property
\[
\left(\mc{E}^{(h)} v,v\right)_{\dom}=\left(v,(I-\Pi_{\Vh})\mc{R}(I+\Pi_{\Vh})v\right)_{\dom}\quad \text{ for all } \quad v\in L^2(D),
\]
where we use the operator $\Pi_{\Vh}:L^{2}(\dom) \to \Vh$ from \eqref{eq:approxLinfty}.
\end{lemma}
A direct consequence of 
\cref{lemma:errOper} is the upper bound for the operator norm of $\mc{E}^{(h)}$
\begin{align}\label{eq:E_{h}}
\norm{\mc{E}^{(h)}}_{\mc{B}(L^2(\dom))}\leq 2C_{\Pi_{\Vh;L^2}} h^{s}\norm{\mc{R}}_{\mc{B}(L^2(\dom), H^s(\dom))},
\end{align}
where we used \eqref{eq:approxLinfty}.

Finally, we are ready to present the main result of this section.
\begin{proposition}[Conforming Galerkin approximation estimate]\label{prop:babuska}
Assume \cref{A:11} holds for $s>d/2$.  Then there are constants $C_1$, $C_2$ and $h_0$ such that
\[
\left| \lambda_{\kllevelrun}^{(h)}-\lambda_{\kllevelrun}\right|\leq C_1\lambda_{\kllevelrun}^{-1}h^{2s} \quad\text{ for all }\quad 0<h\leq h_0.\]
Furthermore, the eigenvectors $\{\phi_{\kllevelrun}^{(h)}\}_{\kllevelrun=1}^{\dimVh}$ of $\mc{R}^{(h)}$ can be chosen so that
\[
\normL{\phi_{\kllevelrun}^{(h)}-\phi_{\kllevelrun}}{\dom}\leq C_{2}\lambda_{\kllevelrun}^{-1} h^{s}\quad\text{ for all }\quad 0<h\leq h_0.\]
Here the constants $C_1$ and $C_2$ are independent of $h$ and $h_0>0$ should be sufficiently small.
\end{proposition}
\begin{proof}The proof follows from \cite[Theorem 9.1]{babuska&osborn91}.
\end{proof}

\subsection{Approximation of the covariance operator from samples}\label{subsec:approxcovhM}
We will now approximate the spectral decomposition \eqref{mercerh} based on the unknown stiffness matrix $\stiffnessmatrix$ by approximating the true covariance matrix $\vec{\Sigma}_{\vec{\lnrfrv}^{(h)}}\in \R^{\dimVh \times \dimVh}$.
Suppose we have such an estimate $\vec{\Sigma}^{(h;\nsamples)}_{\vec{\lnrfrv}^{(h)}}\in \R^{\dimVh \times \dimVh},$
where $\vec{\Sigma}^{(h;\nsamples)}_{\vec{\lnrfrv}^{(h)}}$ is also symmetric. Then, by repeating the approximation for the derivation  \eqref{mercerh}, we obtain the approximations 
\begin{align}
\stiffnessmatrixM&:=\left(\massmatrix\right)^{\transpose}\vec{\Sigma}^{(h;\nsamples)}_{\vec{\lnrfrv}^{(h)}} \massmatrix  \approx \stiffnessmatrix=\left(\massmatrix\right)^{\transpose}\vec{\Sigma}^{(h)}_{\vec{\lnrfrv}^{(h)}} \massmatrix  \label{stiffnessample}\quad \text{and} \quad\\
\tildestiffnessmatrixM&:=\left(\CholStiffnessL\right)^{\transpose}\vec{\Sigma}^{(h;\nsamples)}_{\vec{\lnrfrv}^{(h)}}
\CholStiffnessL \approx  \tildstiffnessmatrix=\left(\CholStiffnessL \right)^{\transpose}\vec{\Sigma}^{(h)}_{\vec{\lnrfrv}^{(h)}}\CholStiffnessL  \label{tildestiffnessample},
\end{align}
and we encounter the following eigenvalue problem: Search $(\widetilde{\vec{\Phi}}^{(h;\nsamples)}_{\kllevelrun},\lambda^{(h;\nsamples)}_{\kllevelrun} ) \in \R^{\dimVh} \times \R$ for all $\kllevelrun=1,\cdots,\dimVh$ such that
\begin{align}\label{genalizedeigenvalue3-estimator}
 \tildestiffnessmatrixM\widetilde{\vec{\Phi}}^{(h;\nsamples)}_{\kllevelrun} = \lambda^{(h;\nsamples)}_{\kllevelrun} \widetilde{\vec{\Phi}}^{(h;\nsamples)}_{\kllevelrun}.
\end{align}
Since $\tildestiffnessmatrixM$ is a symmetric matrix, we have $
\left(\widetilde{\vec{\Phi}}^{(h;\nsamples)}_{\kllevelrun} ,\widetilde{\vec{\Phi}}^{(h;\nsamples)}_{\kllevelrun^{\prime}}  \right)_{\ell_2}=\delta_{\kllevelrun, \kllevelrun^{\prime}}$  after normalization.
We also encounter the generalized eigenvalue problem
\begin{align}\label{genalizedeigenvalue4-estimator}
 \stiffnessmatrixM  \vec{\Phi}^{(h;\nsamples)}_{\kllevelrun} = \lambda^{(h;\nsamples)}_{\kllevelrun} \massmatrix \vec{\Phi}^{(h;\nsamples)}_{\kllevelrun}.
\end{align}
Similarly, we have the equivalence 
\begin{align*}
&(\widetilde{\vec{\Phi}}^{(h;\nsamples)}_{\kllevelrun},\lambda^{(h;\nsamples)}_{\kllevelrun}) \in \R^{\dimVh} \times \R
\text{ is an eigenpair of \eqref{genalizedeigenvalue3-estimator}}  \Longleftrightarrow \\ 
&\left( \vec{\Phi}^{(h;\nsamples)}_{\kllevelrun}:=\left(\CholStiffnessL\right)^{-\transpose}\widetilde{\vec{\Phi}}^{(h;\nsamples)}_{\kllevelrun},
\lambda^{(h;\nsamples)}_{\kllevelrun}\right) \in \R^{\dimVh} \times \R
\text{ is an eigenpair of \eqref{genalizedeigenvalue4-estimator}} .
\end{align*}
Thus, we can derive approximations to $\tildstiffnessmatrix$ and $\stiffnessmatrix$ by
\begin{align}\label{stiffnessmatrixShhat}
\tildestiffnessmatrixM=\sum_{\dimVhrun=1}^{\dimVh} \lambda^{(h;\nsamples)}_{\dimVhrun}\widetilde{\vec{\Phi}}^{(h;\nsamples)}_{\dimVhrun}\otimes \widetilde{\vec{\Phi}}^{(h;\nsamples)}_{\dimVhrun} \text{ and }\stiffnessmatrixM=\sum_{\dimVhrun=1}^{\dimVh} \lambda^{(h;\nsamples)}_{\dimVhrun}\vec{\Phi}^{(h;\nsamples)}_{\dimVhrun}\otimes \vec{\Phi}^{(h;\nsamples)}_{\dimVhrun}.
\end{align}
We define
\begin{align}\label{eq:errorstiffnessh}
\mathcal{E}_{(h;M)}:=\left\| \tildstiffnessmatrix-\tildestiffnessmatrixM \right\|_{2\to 2}
\end{align}
Then the eigensystem \eqref{stiffnessmatrixShhat} gives a computable Mercer-like expansion
\begin{align}\label{mercerh2}
R^{(h;\nsamples)}(\vec{x},\vec{x}^{\prime}) :=\sum_{\dimVhrun=1}^{\dimVh}\lambda^{(h;\nsamples)}_{\dimVhrun}
\phi^{(h;\nsamples)}_{\dimVhrun}(\vec{x}) \, \phi^{(h;\nsamples)}_{\dimVhrun}(\vec{x}^{\prime}).
\end{align}
We define
$\phi_{k}^{(h;\nsamples)}:=\vec{\Phi}^{(h;\nsamples)}_{k} \cdot \vec{\theta}^{(h)}$ 
for all $k=1,\cdots,\dimVh$.
With $
\massmatrix=\CholStiffnessL\left(\CholStiffnessL\right)^{\transpose} \in \R^{\dimVh \times \dimVh}
$, we observe 
that by construction
\begin{align}\label{philhmortho}
\left( \phi_{k}^{(h;\nsamples)}, \phi_{k^{\prime}}^{(h;\nsamples)}\right)_{L_2(\dom)}
=\left(\widetilde{\vec{\Phi}}^{(h;\nsamples)}_{k}, \widetilde{\vec{\Phi}}^{(h;\nsamples)}_{k^{\prime}}\right)_{\ell_2}=\delta_{k,k^{\prime}}.
\end{align}
To bound the sampling error for the eigenvalues, we can
apply Weyl's inequality (see for example \cite[(Eq. 3.1)]{IpsenNadler09}) to either of the eigenvalue problems  \eqref{genalizedeigenvalue3-estimator} or \eqref{genalizedeigenvalue4-estimator}. We will focus on the analysis on the regular eigenvalue problem \eqref{genalizedeigenvalue3-estimator}. 
We obtain
\begin{align}
\label{weyleigenvalue}
\left| \lambda^{(h)}_{\kllevelrun} -\lambda^{(h;\nsamples)}_{\kllevelrun} \right| \le 
\mathcal{E}_{(h;M)}, \quad 1\le  \kllevelrun \le \min\{\dimVh,\kllevel\}.
\end{align}
Now we can use \cref{prop:babuska} to get that, under 
\cref{A:11}, there are constants $C_1$ and $h_0$ such that
\begin{align}\label{weyleigenvalue2}
\left| \lambda_{\kllevelrun} -\lambda^{(h;\nsamples)}_{\kllevelrun} \right|&\le 
\left| \lambda_{\kllevelrun}-\lambda^{(h)}_{\kllevelrun}\right| + \left| \lambda^{(h)}_{\kllevelrun} -\lambda^{(h;\nsamples)}_{\kllevelrun} \right|\nonumber 
\leq C_1\lambda_{\kllevelrun}^{-1}h^{2s}+\mathcal{E}_{(h;M)} 
\end{align}
for all $0<h\leq h_0$ and $ 1 \le \kllevelrun \le \min\{\dimVh,\kllevel\}$.
For perturbation bounds on the eigenvectors, see \cite{varah70}, and for a recent version of the Davis-Kahan theorem including randomness, see  \cite{orourke18}.
We use the variant presented in \cite[Corollary 1]{yu15}. To do this, recall the assumption that the eigenvalues are  non-increasingly ordered, i.e. $\lambda^{(h)}_{1}\ge \dots \ge \lambda^{(h)}_{\dimVh}$ and $\lambda^{(h;\nsamples)}_{1}\ge \dots \ge \lambda^{(h;\nsamples)}_{\dimVh}$.
We define the discrete spectral gap 
\begin{equation}
\label{spectralgap1}
\delta^{(h;\nsamples)}_{\kllevelrun}:=\min \left\{\left|\lambda^{(h;\nsamples)}_{\kllevelrun-1}- \lambda^{(h)}_{\kllevelrun}\right|, \left|\lambda^{(h)}_{\kllevelrun}- \lambda^{(h;\nsamples)}_{\kllevelrun+1}\right|\right\},
\end{equation}
with $\lambda^{(h;\nsamples)}_{0}=\infty$ and $\lambda^{(h;\nsamples)}_{\dimVh+1}=-\infty$.
The quantity $\delta_{\kllevelrun}^{(h;\nsamples)}$ is problematic because it contains $\lambda^{(h)}_{\kllevelrun}, \lambda^{(h;\nsamples)}_{\kllevelrun\pm 1}$. Therefore, we want to replace it with a quantity that depends only on the continuous problem. To do this, we use the strategy presented in \cite{yu15}:
Let $\delta_{\kllevelrun}$ be the continuous spectral gap defined by
\begin{align}\label{defn:gap}
\delta_{\kllevelrun}:=\min \left\{\lambda_{\kllevelrun-1}- \lambda_{\kllevelrun}, \lambda_{\kllevelrun}- \lambda_{\kllevelrun+1}\right\} \text{ with } \lambda_0=\infty.
\end{align}

\begin{assumption}[Spectral gap]\label{ass:spectralgap}
Assume that there is a sufficiently small positive parameter $h_{1}\le h_0$ and that
$\tildestiffnessmatrixM$ is a good approximation of $\tildstiffnessmatrix$, i.e. for $h\le h_1$ holds
\begin{align}
\label{condspectralgap1}
\delta_{\kllevelrun}\ge 4C_1h^{2s} \lambda_{\kllevelrun+1}^{-1} +4\mathcal{E}_{(h;M)}.
\end{align}
\end{assumption}
We then have the following result:
\begin{theorem}\label{theoremeigenfucntions}
Let  \eqref{condspectralgap1} be valid. It then holds
\begin{align}\label{spectralgap3}
\delta^{(h;\nsamples)}_{\kllevelrun} \ge \frac{1}{4} \delta_{\kllevelrun}.
\end{align}

With some constant $C$, it also holds 
\begin{align}
\label{kahandaviscor}
\left\| \widetilde{\vec{\Phi}}^{(h)}_{\kllevelrun}- \widetilde{\vec{\Phi}}^{(h;\nsamples)}_{\kllevelrun} \right\|_2\le C (\delta^{(h;\nsamples)}_{\kllevelrun})^{-1}\mathcal{E}_{(h;M)} \le 4C\delta^{-1}_{\kllevelrun} \mathcal{E}_{(h;M)},
\end{align}
where we fixed a sign for $\widetilde{\vec{\Phi}}^{(h;\nsamples)}_{\kllevelrun}$ such that $\widetilde{\vec{\Phi}}^{(h)}_{\kllevelrun}\cdot \widetilde{\vec{\Phi}}^{(h;\nsamples)}_{\kllevelrun}\ge 0$.

Also for $\phi_{\kllevelrun}^{(h;\nsamples)}:=\left(\CholStiffnessL\right)^{\transpose} \widetilde{\vec{\Phi}}^{(h;\nsamples)}_{k} \cdot \vec{\theta}^{(h)}$ the bound
\begin{align}
\label{phivecerror}
\left\|\phi_{\kllevelrun}^{(h)}- \phi_{\kllevelrun}^{(h;\nsamples)} \right\|_{L_{2}(\dom)}\le C (\delta^{(h;\nsamples)}_{\kllevelrun})^{-1}\mathcal{E}_{(h;M)}\le 4C \delta^{-1}_{\kllevelrun}\mathcal{E}_{(h;M)}
\end{align}
holds.
\end{theorem}
\begin{proof}
The first statement is a generalization of an argument presented in \cite{yu15} to the case of multiple discretization parameters.
Recall the discrete spectral gap $\delta^{(h;\nsamples)}_{\kllevelrun}$ in \eqref{spectralgap1}. 
We show \eqref{spectralgap3} using the following strategy:
First we observe
$
\delta_{\kllevelrun}
\le \left|\lambda_{\kllevelrun-1}-\lambda^{(h)}_{\kllevelrun-1}\right|+\left|\lambda^{(h)}_{\kllevelrun-1}-\lambda^{(h)}_{\kllevelrun}\right|+\left|\lambda^{(h)}_{\kllevelrun}- \lambda_{\kllevelrun} \right|
$
by the triangle inequality.
Similarly, we get
$
\delta_{\kllevelrun}
\le \left|\lambda_{\kllevelrun}-\lambda^{(h)}_{\kllevelrun}\right|+\left|\lambda^{(h)}_{\kllevelrun}-\lambda^{(h)}_{\kllevelrun+1}\right|+\left|\lambda^{(h)}_{\kllevelrun+1}- \lambda_{\kllevelrun+1} \right|.
$
Using \cref{prop:babuska} and the fact that the eigenvalues are sorted, i.e. $\lambda_{\kllevelrun-1} \ge \lambda_{\kllevelrun}\ge \lambda_{\kllevelrun+1}$, we obtain for all $0<h\leq h_0$ the bound
\begin{align*}
\max\left\{\left|\lambda_{\kllevelrun-1}-\lambda^{(h)}_{\kllevelrun-1}\right|,\left|\lambda^{(h)}_{\kllevelrun}- \lambda_{\kllevelrun} \right|,\left|\lambda_{\kllevelrun+1}-\lambda^{(h)}_{\kllevelrun+1}\right| \right\} \le C_1 \lambda_{\kllevelrun+1}^{-1}h^{2s} \le  \frac{1}{4} \delta_{\kllevelrun},
\end{align*}
where we used the spectral gap assumption \eqref{condspectralgap1} in the last step.  
So we derive \\
$
\delta_{\kllevelrun} \le \min \left\{ \left|\lambda^{(h)}_{\kllevelrun-1}-\lambda^{(h)}_{\kllevelrun}\right|,\left|\lambda^{(h)}_{\kllevelrun}-\lambda^{(h)}_{\kllevelrun+1}\right|\right\} + \frac{1}{2}\delta_{\kllevelrun} = \delta^{(h)}_{\kllevelrun} + \frac{1}{2}\delta_{\kllevelrun}.
$
This inequality implies
\begin{align}\label{spectralgap_auxiliary1}
\frac{1}{2} \delta_{\kllevelrun} \le \delta^{(h)}_{\kllevelrun} =\min
\left\{ \left|\lambda^{(h)}_{\kllevelrun-1}-\lambda^{(h)}_{\kllevelrun}\right|,\left|\lambda^{(h)}_{\kllevelrun}-\lambda^{(h)}_{\kllevelrun+1}\right|\right\}.
\end{align}
We have  
$
\delta^{(h)}_{\kllevelrun}
\le \left|\lambda^{(h)}_{\kllevelrun-1}-\lambda^{(h;\nsamples)}_{\kllevelrun-1}\right|+\left|\lambda^{(h;\nsamples)}_{\kllevelrun-1}- \lambda^{(h)}_{\kllevelrun} \right|
$
and
$
\delta^{(h)}_{\kllevelrun}
\le \left|\lambda^{(h)}_{\kllevelrun}-\lambda^{(h;\nsamples)}_{\kllevelrun}\right|+\left|\lambda^{(h;\nsamples)}_{\kllevelrun}- \lambda^{(h)}_{\kllevelrun+1} \right|
$
by the triangle inequality.
Using Weyl's Theorem, i.e. \eqref{weyleigenvalue}, and the spectral gap assumption \eqref{condspectralgap1}, we get
\begin{align*}
\max\left\{ \left|\lambda^{(h)}_{\kllevelrun-1}-\lambda^{(h;\nsamples)}_{\kllevelrun-1}\right|,\left|\lambda^{(h)}_{\kllevelrun}-\lambda^{(h;\nsamples)}_{\kllevelrun}\right| \right\}\le \left\| \tildstiffnessmatrix-\tildestiffnessmatrixM \right\|_{2 \to 2} \le \frac{1}{4}\delta_{\kllevelrun} \le \frac{1}{2}\delta^{(h)}_{\kllevelrun},
\end{align*}
where we used \eqref{spectralgap_auxiliary1} in the last step. So we get 
\begin{align*}
\delta^{(h;\nsamples)}_{\kllevelrun}= \min\left\{ \left|\lambda^{(h)}_{\kllevelrun}-\lambda^{(h;\nsamples)}_{\kllevelrun-1}\right|,\left|\lambda^{(h)}_{\kllevelrun+1}-\lambda^{(h;\nsamples)}_{\kllevelrun}\right| \right\}\ge \frac{1}{2}\delta^{(h)}_{\kllevelrun} \ge \frac{1}{4}\delta_{\kllevelrun},
\end{align*}
with \eqref{spectralgap_auxiliary1} again. We also have 
$
\left\|\phi_{\kllevelrun}^{(h)}- \phi_{\kllevelrun}^{(h;\nsamples)} \right\|^2_{L_{2}(\dom)}
=  \left\| \widetilde{\Phi}^{(h)}_{\kllevelrun}-\widetilde{\Phi}^{(h;\nsamples)}_{\kllevelrun}\right\|^2_{\ell_2}.
$
The other two statements follow directly from the Davis-Kahan theorem, as in \cite[Corollary1]{yu15}.
\end{proof}
Finally, we are left with bounding
\begin{align}\label{tildeSopnorm}
\mathcal{E}_{(h;M)}=\left\| \tildstiffnessmatrix-\tildestiffnessmatrixM \right\|_{2\to 2}
=\sup_{\vec{y} \in \R^{\dimVh}\setminus \{ \vec{0}\}} \frac{\left( \left( \vec{\Sigma}_{\vec{\lnrfrv}^{(h)}}- \vec{\Sigma}^{(h;\nsamples)}_{\vec{\lnrfrv}^{(h)}}\right) \vec{y},\vec{y}\right)_{\ell_2}}{\|
\left(\CholStiffnessL\right)^{-1}  \vec{y}\|_{\ell_2}^2},
\end{align}
where we use the identity
$
\tildstiffnessmatrix-\tildestiffnessmatrixM:=\left(\CholStiffnessL\right)^{\transpose}\left(\vec{\Sigma}_{\vec{\lnrfrv}^{(h)}}- \vec{\Sigma}^{(h;\nsamples)}_{\vec{\lnrfrv}^{(h)}}\right)\CholStiffnessL 
$
and the fact that $ \vec{y}=\CholStiffnessL \vec{x}$ is admissible since $\CholStiffnessL$ is regular.
Note here that\\
$
\lambda_{\min}\left(\left(\massmatrix\right)^{-1}\right) \|\vec{y}\|^{2} \le \left(\left( \massmatrix\right)^{-1}\vec{y},\vec{y} \right) \le \lambda_{\max}\left(\left(\massmatrix\right)^{-1}\right)\|\vec{y}\|^{2}
$ holds.  
We now have the bound
\begin{align}\label{lieqneu}
	\mathcal{E}_{(h;M)} \in \left[ \lambda_{\min}\left(\massmatrix\right), \lambda_{\max}\left(\massmatrix\right) \right] \left\| \vec{\Sigma}_{\vec{\lnrfrv}^{(h)}}- \vec{\Sigma}^{(h;\nsamples)}_{\vec{\lnrfrv}^{(h)}}  \right\|_{2\to 2} .
\end{align}
The matrix norm $ \left\| \vec{\Sigma}_{\vec{\lnrfrv}^{(h)}}- \vec{\Sigma}^{(h;\nsamples)}_{\vec{\lnrfrv}^{(h)}}  \right\|_{2\to 2}$
is estimated in probabilistic terms in 
\cref{subsec:approxcovhMfinal}.

\subsection{Sub-Gaussian tails for observed random vectors}\label{sec:tails}
In order to use recent results to bound the sampling error, we must first consider the sub-Gaussian property of the random  vector $\vec{\lnrfrv}^{(h)}$. 
This random vector consists of the coefficients of the discrete random field $\lnrf^{(h)}:\Omega \times \dom \to \R$, i.e. 
$\lnrf^{(h)}(\vec{\om},\vec{x})= \sum_{\dimVhrun=1}^{\dimVh} \lnrfrv^{(h)}_{\dimVhrun}(\vec{\om}) \theta_{\dimVhrun}^{(h)}(\vec{x}) = \vec{\lnrfrv}^{(h)}(\vec{\om}) \cdot \vec{\theta}(\vec{x}).
$
We focus mainly on two different basis functions. First, we choose $\{\theta^{(h)}_{\dimVhrun}\}_{1\le \dimVhrun \le \dimVh}$ to be a nodal basis, and second, we choose $\{\theta^{(h)}_{\dimVhrun}\}_{1\le \dimVhrun \le \dimVh}$ to be an $L^{2}$-orthonormal basis.
\subsubsection{Nodal basis}
To start, let us consider the setting of the so-called standard information. Recall the point set $X_{\dimVh}=\left\{ \vec{x}_1,\dots,\vec{x}_{\dimVh}\right\}\subset \dom$, which determines the nodal basis functions and thus the degrees of freedom for the finite element space $\Vh$.
In this case, 
$
\lnrf^{(h)}(\vec{\om},\vec{x}) = \vec{\lnrfrv}^{(h)}(\vec{\om}) \cdot \vec{\theta}(\vec{x})=\sum_{\dimVhrun=1}^{\dimVh} \lnrf^{(h)}_{\dimVhrun}(\vec{\om},\vec{x}_{\dimVhrun})\theta_{\dimVhrun}^{(h)}(\vec{x}).
$
So we have
\begin{align}\label{observedrandomvectornodal}
\vec{\lnrfrv}^{(h)}: \Omega \to \R^{\dimVh}, \quad \vec{\om} \mapsto \vec{\lnrfrv}^{(h)}(\vec{\om})=\left(\lnrf^{(h)}(\vec{\om},\vec{x}_1),\dots,\lnrf^{(h)}(\vec{\om},\vec{x}_{\dimVh}) \right)^{\transpose}.
\end{align}
Since $\lnrf^{(h)}$ is assumed to be a centered Gaussian random field, it follows by definition that the random vector $\vec{\lnrfrv}^{(h)}$ is also  distributed according to a multivariate Gaussian law, i.e. 
$\vec{\lnrfrv}^{(h)}\sim \mathcal{N}\left(\vec{0},\vec{\Sigma}_{\vec{K}^{(h)}} \right)$
with \eqref{truecovariance}, i.e. 
$
\vec{\Sigma}_{\vec{K}^{(h)}}
=\expect{\vec{\lnrfrv}^{(h)} \otimes \vec{\lnrfrv}^{(h)} }.
$
We can invoke \cref{A:00} and \eqref{truecovariance}, i.e. $\left( \vec{\Sigma}_{\vec{K}^{(h)}}\right)_{\dimVhrun,\dimVhrun^{\prime}} \le 4(\constC^{(h)})^2$ . 
Finally, by Chernoff's inequality \cite{cherstat} we get the bound
$$\prob\left\{  \vec{v} \cdot\left( \vec{\lnrfrv}^{(h)}-\expect{\vec{\lnrfrv}^{(h)}}\right) > t\right\} \le  \exp\left(-\frac{t^2}{8 (\constC^{(h)})^2} \right)$$
for any $\vec{v} $ with $\| \vec{v} \|_2 =1$.
Note that this implies {$4 (\constC^{(h)})^2=\rho^{-1}$} in the sense of \cite[Definition 1]{cai2016}.
\subsubsection{$L^2$-orthonormal basis}
Now let us consider the setting of so-called linear information. Here, we cannot directly rely on the definition of a centered Gaussian random field to infer that the coefficients are distributed with a multivariate normal law.
In this case, we have
\begin{align*}
\vec{\lnrfrv}^{(h)}:\Omega \to \R^{\dimVh}, \quad \vec{\om} &\mapsto \left(  (\lnrf^{(h)} (\vec{\om},\cdot) ,\theta^{(h)}_{1})_{\dom} , \dots, ( \lnrf^{(h)} (\vec{\om},\cdot) ,\theta^{(h)}_{\dimVh} )_{\dom}\right)^{\transpose},
\end{align*}
where $\{\theta_j^{(h)}, j=1,\ldots,\dimVh\}$ is an $L_2$-orthonormal basis.
Thus, for the expected value we observe $\expect{\lnrfrv^{(h)}_{\dimVhrun}}=0$
which implies $\expect{\vec{\lnrfrv}^{(h)}}=0$.
For the variance we have again \eqref{truecovariance}, i.e. 
$
\vec{\Sigma}_{\vec{K}^{(h)}}=\expect{\vec{\lnrfrv}^{(h)} \otimes \vec{\lnrfrv}^{(h)} }$,
where
\begin{equation}\label{li2neu}
	\left(\vec{\Sigma}_{\vec{K}^{(h)}}\right)_{k,k^{\prime}} =\int_{\dom} \int_{\dom}
	\text{Cov}_{\lnrf^{(h)}}(\vec{x},\vec{x}^{\prime}) \theta_{k}^{h}(\vec{x})\theta_{k^{\prime}}^{h}(\vec{x}^{\prime}) \mathrm{d}\vec{x} \mathrm{d} \vec{x}^{\prime}
\end{equation}
and thus
$\left(\vec{\Sigma}_{\vec{K}^{(h)}}\right)_{k,k^{\prime}} \leq 
	\|\text{Cov}_{\lnrf^{(h)}}\|_{L^2(\dom \times \dom)} \leq 4 (\constC^{(h)})^2 \cdot | \dom|.$
For all $\vec{v} \in \R^{\dimVh}$ with $\| \vec{v}\|_2=1$, we have
\begin{align*}
\vec{v}\cdot \left(\vec{\lnrfrv}^{(h)}-\expect{\vec{\lnrfrv}^{(h)}}\right)
	&=\sum_{k=1}^{\dimVh} \left(\lnrf^{(h)}(\om_j,\cdot), \theta_j^{(h)} \right)_{L^2(\dom)} v_j \sim \mathcal{N} (0,\vec{v}\cdot \vec{\Sigma}_{\vec{K}^{(h)}}\vec{v}).
\end{align*}
Furthermore, given $\vec{x} \in \dom$, it holds that
$
\lnrf^{(h)}(\cdot,\vec{x}) \sim \mathcal{N}\left(0,\text{Cov}_{\lnrf^{(h)}}(\vec{x},\vec{x})\right).
$
Therefore, we can invoke Chernoff's inequality to obtain for all $\vec{x} \in \dom$ and $t>0$
\begin{align*}
\prob \left\{ \ \left| \lnrf^{(h)} (\omega,\vec{x}) \right| >t  \right\}
&\le  \exp\left(-\frac{t^2}{2\text{Cov}_{\lnrf^{(h)}}(\vec{x},\vec{x})} \right)\le  \exp\left(-\frac{t^2}{8(\constC^{(h)})^2|\dom| } \right),
\end{align*}
which is uniform for $\vec{x}\in \dom$, see \eqref{li2neu} and its consequence. Consequently,  for all $\vec{v} \in \R^{\dimVh} $  with  $\| \vec{v}\|_2=1$, we have the estimate
	$$\prob\left\{  \left| \vec{v} \cdot\left( \vec{\lnrfrv}^{(h)}-\expect{\vec{\lnrfrv}^{(h)}}\right) \right|  > t\right\} \le \exp\left(-\frac{t^2}{{8  (\constC^{(h)}})^2 |\dom| } \right).$$
Note that this implies $4(\constC^{(h)})^2 |\dom| 
 =\rho^{-1} $ in the sense of \cite[Definition 1]{cai2016}.
Thus, the random vector $\vec{\lnrfrv}^{(h)}$ obeys a sub-Gaussian bound in the sense of \cite[Eq. (7)]{cai2010}.
We will use this estimate together with the results of \cite{cai2010} to bound the variance  of the sampled covariance matrix later.

\subsection{Covariance estimation from samples using tapering and decay assumptions}\label{subsec:approxcovhMfinal}
In this subsection, we focus on the sampling error, and specifically on the induced variance of our estimator for the covariance matrix. We assume now that the parameter $h$ is fixed. Thus, the dimension of the covariance matrix is given by $\dimVh$. We will provide a constructive approach to obtain an estimate of the covariance matrix  $\vec{\Sigma}_{\vec{\lnrfrv}^{(h)}}$. To do this, given samples $\vec{\lnrfrv}^{(\nsamplesrun;h)}$ for $\nsamplesrun=1,\cdots,M$, let the sample mean and the maximum likelihood estimator for the covariance matrix $\vec{\Sigma}_{\vec{\lnrfrv}^{(h)}}$ be \cite[Eq. (2)]{cai2010}
\begin{align}\label{astest}
\bar{\vec{\lnrfrv}}^{(h;\nsamples)}&:=\frac{1}{\nsamples}\sum_{\nsamplesrun=1}^{\nsamples}\vec{\lnrfrv}^{(\nsamplesrun;h)} \in \R^{\dimVh}\quad \text{and} \\
\label{astest2}
\bar{\vec{\Sigma}}^{(h;\nsamples)}_{\vec{\lnrfrv}^{(h)}}&:= \frac{1}{\nsamples} \sum_{\nsamplesrun=1}^{\nsamples} \left(\vec{\lnrfrv}^{(\nsamplesrun;h)}-\bar{\vec{\lnrfrv}}^{(h)}\right)\otimes \left(\vec{\lnrfrv}^{(\nsamplesrun;h)}-\bar{\vec{\lnrfrv}}^{(h)}\right)\in \R^{\dimVh \times \dimVh}.
\end{align}
We know from \cite[Lemma 1]{cai2016} that the  direct covariance estimator \eqref{astest2} suffers from the curse of dimension with respect to high spatial resolution $\dimVh$. Consequently, we need a better estimator when $\dimVh$ is large. To this end, we follow \cite{cai2016} and assume an off-diagonal decay of the covariance matrix $\vec{\Sigma}_{\vec{\lnrfrv}^{(h)}} \in \R^{\dimVh \times \dimVh}$, i.e. we assume  $\vec{\Sigma}_{\vec{\lnrfrv}^{(h)}} \in \decayclass_{\alpha}=\decayclass_{\alpha}(C_{\decayclass;1},C_{\decayclass;2}) \subset \R^{\dimVh \times \dimVh} $, where
\begin{equation}\label{taperingclass}
\decayclass_{\alpha}:=\left\{\vec{\Sigma} \ : \ \max_{\dimVhrun=1,\dots,\dimVh}\sum_{\genfrac{}{}{0pt}{}{\dimVhrun^{\prime}=1}{\left|\dimVhrun^{\prime}-\dimVhrun \right|>c }}^{\dimVh}\left|\vec{\Sigma}_{\dimVhrun,\dimVhrun^{\prime}}\right| \le C_{\decayclass;1} c^{-\alpha}, \  \forall 1\le c \le \dimVh, \  \lambda_{\max}(\vec{\Sigma})\le C_{\decayclass;2}\right\}.
\end{equation}
Here $\lambda_{\max}(\vec{\Sigma})$ is the maximum eigenvalue of $\vec{\Sigma}$, and $\alpha$ modulates the decay rate, while $C_{\decayclass;1}$ and $C_{\decayclass;2}$ are positive parameters.
Now let the weights for any even integer $1\le \taperingint \le \dimVh$ (see \cite[(Eq. 5)]{cai2010}) be given as
\begin{align}\label{taperingweights}
\taperingweight_{\dimVhrun,\dimVhrun^{\prime}}:=\taperingweight_{\dimVhrun,\dimVhrun^{\prime}}(\tau):=
\begin{cases}1, & \left|\dimVhrun-\dimVhrun^{\prime} \right|\le \frac{\taperingint}{2}, \\
2\left(1-\frac{\left|\dimVhrun-\dimVhrun^{\prime} \right|}{\taperingint} \right), &  \frac{\taperingint}{2} < \left|\dimVhrun-\dimVhrun^{\prime} \right|<\taperingint ,\\
0, & \left|\dimVhrun-\dimVhrun^{\prime} \right|\ge \taperingint.
 \end{cases}
\end{align}
Then the associated tapering estimator (c.f. \cite[(Eq. 4)]{cai2010}) is defined element-wise by
\begin{align}
\label{taperingestimator}
\left(\vec{\Sigma}^{(h;\nsamples;\tau)}_{\vec{\lnrfrv}^{(h)}}\right)_{_{\dimVhrun,\dimVhrun^{\prime}}}:=\taperingweight_{\dimVhrun,\dimVhrun^{\prime}}(\tau)\left(\bar{\vec{\Sigma}}^{(h;\nsamples)}_{\vec{\lnrfrv}^{(h)}}\right)_{\dimVhrun,\dimVhrun^{\prime}}, \quad 1 \le \dimVhrun,\dimVhrun^{\prime} \le \dimVh.
\end{align}
Note that the estimator $\vec{\Sigma}^{(h;\nsamples;\tau)}_{\vec{\lnrfrv}^{(h)}}$ preserves self-adjointness. Furthermore, the bounds
\begin{align}
\expect{\left\|\vec{\Sigma}^{(h;\nsamples;\tau)}_{\vec{\lnrfrv}^{(h)}}-\expect{\vec{\Sigma}^{(h;\nsamples;\tau)}_{\vec{\lnrfrv}^{(h)}}} \right\|_{2 \to 2}^2}&\lesssim \frac{\taperingint+\log(\dimVh)}{\nsamples}\label{taperingerrorvar},\\
\left\|\expect{\vec{\Sigma}^{(h;\nsamples;\tau)}_{\vec{\lnrfrv}^{(h)}}}-\vec{\Sigma}_{\vec{\lnrfrv}^{(h)}} \right\|_{2\to 2}^2&\lesssim \taperingint^{-2\alpha} \label{taperingerrorbias}
\end{align}
on the variance and the bias hold, see \cite[(Eqs. 13 \& 14)]{cai2010}.
From the triangle inequality we derive the convergence of $\vec{\Sigma}^{(h;\nsamples;\tau)}_{\vec{\lnrfrv}^{(h)}}\to \vec{\Sigma}_{\vec{\lnrfrv}^{(h)}}$ in expectation as
\begin{align}\label{taperingerrorbiasplusvar}
&\expect{\left\|\vec{\Sigma}^{(h;\nsamples;\tau)}_{\vec{\lnrfrv}^{(h)}}-\vec{\Sigma}_{\vec{\lnrfrv}^{(h)}}\right\|_{2\to 2}^2} \le 2\left\|\expect{\vec{\Sigma}^{(h;\nsamples;\tau)}_{\vec{\lnrfrv}^{(h)}}}-\vec{\Sigma}_{\vec{\lnrfrv}^{(h)}} \right\|_{2\to 2}^2 \\&+2\expect{\left\|\vec{\Sigma}^{(h;\nsamples;\tau)}_{\vec{\lnrfrv}^{(h)}}-\expect{\vec{\Sigma}^{(h;\nsamples;\tau)}_{\vec{\lnrfrv}^{(h)}}} \right\|_{2\to 2}^2}\nonumber \lesssim \frac{\taperingint+\log(\dimVh)}{\nsamples}+\taperingint^{-2\alpha},
\end{align}
which motivates the choice of the tapering parameter $\tau=\nsamples^{\frac{1}{2\alpha+1}}$, see \cite[(Eq. 15)]{cai2010}.
This gives the following result, see \cite[Theorem 2 \& (Eq. 31)]{cai2010}:

\begin{proposition}\label{propsamplingexpec}
For the tapering estimator \eqref{taperingestimator} with $\tau=\nsamples^{\frac{1}{2\alpha+1}}$, i.e.
\begin{align}\label{optimaltaperingchoice}
	\left(\vec{\Sigma}^{(h;\nsamples)}_{\vec{\lnrfrv}^{(h)}}\right)_{_{\dimVhrun,\dimVhrun^{\prime}}}:=\taperingweight_{\dimVhrun,\dimVhrun^{\prime}}(\nsamples^{\frac{1}{2\alpha+1}} )\left(\bar{\vec{\Sigma}}^{(h;\nsamples)}_{\vec{\lnrfrv}^{(h)}}\right)_{\dimVhrun,\dimVhrun^{\prime}}, \quad 1 \le \dimVhrun,\dimVhrun^{\prime} \le \dimVh,
\end{align}
there holds the error estimate 
\begin{align}\label{taperingerrorbiasplusvaroptimal}
\expect{\left\|\vec{\Sigma}^{(h;\nsamples)}_{\vec{\lnrfrv}^{(h)}}-\vec{\Sigma}_{\vec{\lnrfrv}^{(h)}}\right\|_{2\to 2}^2} \lesssim \nsamples^{-\frac{2\alpha}{2\alpha+1}} + \frac{\log(\dimVh)}{\nsamples}
\end{align}
if the condition $\dimVh \ge  \nsamples^{\frac{1}{2\alpha+1}}$ is satisfied.
In the case  $\dimVh < \nsamples^{\frac{1}{2\alpha+1}}$, we can directly use the estimator from \eqref{astest2}
and get the bound
\begin{align}\label{taperingerrorbiasplusvaroptimal2}
\expect{\left\|\bar{\vec{\Sigma}}^{(h;\nsamples)}_{\vec{\lnrfrv}^{(h)}}-\vec{\Sigma}_{\vec{\lnrfrv}^{(h)}}\right\|_{2\to 2}^2} \lesssim \frac{\dimVh }{\nsamples}.
\end{align}
\end{proposition}
\noindent Note that the result is actually \emph{rate optimal}.

Next, we introduce the notation
\begin{align}\label{rhoM}
\rho_{h}(\nsamples)&:=\begin{cases}\nsamples^{-\frac{2\alpha}{2\alpha+1}} + \frac{\log(\dimVh)}{\nsamples}, &  \dimVh \ge \nsamples^{\frac{1}{2\alpha+1}} \\  \frac{\dimVh }{\nsamples}, & \dimVh < \nsamples^{\frac{1}{2\alpha+1}} \end{cases} \\ \label{tilderhoM}& \lesssim \tilde{\rho}_{h}(\nsamples):=\begin{cases}\nsamples^{-\frac{2\alpha}{2\alpha+1}} + \frac{d\log(h^{-1})}{\nsamples}, &  \dimVh \ge \nsamples^{\frac{1}{2\alpha+1}} \\  \frac{h^{-d} }{\nsamples}, & \dimVh < \nsamples^{\frac{1}{2\alpha+1}} \end{cases} ,
\end{align}
using $\dimVh\sim s^d h^{-d}$ for the definition of $\tilde{\rho}_{h}(\nsamples)$.
Then we have the following corollary:
\begin{corollary}\label{corboundtildeS}
For the tapering estimator \eqref{taperingestimator} with $\tau=\nsamples^{\frac{1}{2\alpha+1}}$, the bound
\begin{align}
\label{errorexpec}
\expect{ \mathcal{E}_{(h;M)}^{k}} &\lesssim \rho^{\frac{k}{2}}_{h}(\nsamples) ~ \lambda^{k}_{\max}\left(\massmatrix\right), \quad k \in \{1,2\}
\end{align}
holds.
\end{corollary}
\begin{proof}
	First we get by \eqref{lieqneu} 
\begin{align}\label{eq:zzz1}
 \mathcal{E}_{(h;M)}  \le \lambda_{\max}\left(\massmatrix\right) 
	\left\| \vec{\Sigma}^{(h;\nsamples)}_{\vec{\lnrfrv}^{(h)}}-\vec{\Sigma}_{\vec{\lnrfrv}^{(h)}}  \right\|_{2 \to 2} .
\end{align}
We can use Jensen's inequality for the convex function $\psi(x)=x^2$ to get
$
\psi\left( \expect{X}\right) \le \expect{\psi(X)}.
$
Plugging \eqref{eq:zzz1} into Jensen's inequality  yields
\begin{align*}
\left(\expect{\mathcal{E}_{(h;M)}}\right)^2 \le 
\expect{ \left\| \left(\vec{\Sigma}^{(h;\nsamples)}_{\vec{\lnrfrv}^{(h)}}-\vec{\Sigma}_{\vec{\lnrfrv}^{(h)}} \right)  \right\|^2_{2 \to 2} } \lambda_{\max}^2 \left(\massmatrix\right) .
\end{align*}
So we get from \cref{propsamplingexpec} and with the definition \eqref{rhoM} the bound
\begin{align*}
\expect{\mathcal{E}_{(h;M)}} \lesssim \rho^{\frac{1}{2}}_{h}(\nsamples)~\lambda_{\max}\left(\massmatrix\right) .
\end{align*}
Another application of Jensen's inequality completes the proof.
\end{proof}

Next, we need to verify that Assumption \eqref{condspectralgap1}
holds, at least with high probability. Now, we have the following result:
\begin{proposition}\label{prop:spectralgap}
Let the meshsize $h$ be sufficiently small and $\nsamples$ sufficiently large, 
so that for all $1\le \kllevelrun \le \kllevel$ we have
\begin{align}\label{condspectralgapfinal1}
  \frac{\delta_{\kllevelrun}}{2}\ge 4C_1 h^{2s} \lambda_{\kllevel+1}^{-1}   \text{ and } 
  C \nsamples^{-\frac{2\alpha}{2\alpha+1}} \le \frac{\min_{1\le \kllevelrun \le \kllevel }\delta_{\kllevelrun}}{16 ~ \lambda_{\max}\left(\massmatrix\right) } .
\end{align}
Then we have with probability
\begin{align}\label{eq:ppp}
	p_0:=1- 2 \dimVh 5^{\taperingint}\exp\left(- \nsamples  \rho_1  \left( \frac{\min_{1\le \kllevelrun \le \kllevel }\delta_{\kllevelrun}}{48~\lambda_{\max}\left(\massmatrix\right) } \right)^2\right), \text{ where } \tau = \nsamples^{\frac{1}{2\alpha+1}},
\end{align}
that the condition \eqref{condspectralgap1} in our spectral gap  
\cref{ass:spectralgap} is satisfied.
\end{proposition}
\begin{proof}
Given the estimate \eqref{lieqneu}
it is sufficient to ensure that
\begin{align*}
\delta_{\kllevelrun}\ge 4C_1h^{2s} \lambda_{\kllevelrun+1}^{-1} +4 \left\| \vec{\Sigma}^{(h;\nsamples)}_{\vec{\lnrfrv}^{(h)}}-\vec{\Sigma}_{\vec{\lnrfrv}^{(h)}}   \right\|_{2 \to 2} \lambda_{\max}\left(\massmatrix\right)   \quad \text{for all } 1\le \kllevelrun \le \kllevel.
\end{align*}
The condition \eqref{condspectralgapfinal1} reduces this to the estimate
$\frac{\min_{1\le \kllevelrun \le \kllevel }\delta_{\kllevelrun}}{8~ \lambda_{\max}\left(\massmatrix\right)  }\ge   \left\| \vec{\Sigma}^{(h;\nsamples)}_{\vec{\lnrfrv}^{(h)}}-\vec{\Sigma}_{\vec{\lnrfrv}^{(h)}}   \right\|_{2 \to 2}$.
We want to show that this inequality is only violated with small probability.
To do this, we derive by 
repeating the arguments of \cite[Lemmata 2 \& 3]{cai2010} and \cite[Remark 1]{cai2010} the bound
$
\prob\left\{ \left\|\vec{\Sigma}^{(h;\nsamples)}_{\vec{\lnrfrv}^{(h)}}-\expect{\vec{\Sigma}^{(h;\nsamples)}_{\vec{\lnrfrv}^{(h)}}}\right\|_{2\to 2}>t \right\} \le 2 \dimVh 5^{\taperingint}\exp\left(-\frac{1}{9}\nsamples t^2 \rho_1 \right)
$ 
for all 
$0< t<\rho_1$,
where $\rho_1$ is a constant, see \cite[page 2142]{cai2010}, and $\tau = \nsamples^{\frac{1}{2\alpha+1}}$.
So with \eqref{taperingerrorbias} we get  the triangle inequality, and by setting $\tau=\nsamples^{\frac{1}{2\alpha+1}}$ we get the bound
\begin{equation}\label{covboundinprobab}
\prob\left\{ \left\|\vec{\Sigma}^{(h;\nsamples)}_{\vec{\lnrfrv}^{(h)}}-\vec{\Sigma}_{\vec{\lnrfrv}^{(h)}} \right\|_{2\to 2} >t \right\}\nonumber 
\le \prob\left\{
\left\|\vec{\Sigma}^{(h;\nsamples)}_{\vec{\lnrfrv}^{(h)}}-\expect{\vec{\Sigma}^{(h;\nsamples)}_{\vec{\lnrfrv}^{(h)}}}\right\|_{2\to 2}+C \nsamples^{-\frac{2\alpha}{2\alpha+1}} >t \right\}. 
\end{equation}
Now let $\bar{t} \in (0,\rho_1) $ be fixed and let $\nsamples$ be large enough so that $C \nsamples^{-\frac{2\alpha}{2\alpha+1}} \le \frac{\bar{t}}{2}$. Then it holds that
\begin{align*}
\prob\left\{ \left\|\vec{\Sigma}^{(h;\nsamples)}_{\vec{\lnrfrv}^{(h)}}-\vec{\Sigma}_{\vec{\lnrfrv}^{(h)}} \right\|_{2\to 2} >\bar{t} \right\}&\le 
\prob\left\{
\left\|\vec{\Sigma}^{(h;\nsamples)}_{\vec{\lnrfrv}^{(h)}}-\expect{\vec{\Sigma}^{(h;\nsamples)}_{\vec{\lnrfrv}^{(h)}}}\right\|_{2\to 2} > \frac{\bar{t}}{2} \right\}\\& \le 2 \dimVh 5^{\taperingint}\exp\left(-\frac{1}{36}\nsamples \bar{t}^2 \rho_1 \right). 
\end{align*}
Now we set $\bar{t}:=\frac{1}{8} \min_{1\le \kllevelrun \le \kllevel }\delta_{\kllevelrun}$  and notice that the second inequality in the condition \eqref{condspectralgapfinal1} ensures that 
$C \nsamples^{-\frac{2\alpha}{2\alpha+1}} \le\frac{1}{2}\bar{t} =\frac{1}{16~ \lambda_{\max}\left(\massmatrix\right)  }\min_{1\le  \kllevelrun \le \kllevel }\delta_{\kllevelrun}$. 
This gives us 
\begin{eqnarray}\label{covboundinprobab3}
\prob\left\{ \left\|\vec{\Sigma}^{(h;\nsamples)}_{\vec{\lnrfrv}^{(h)}}-\vec{\Sigma}_{\vec{\lnrfrv}^{(h)}} \right\|_{2\to 2} > \frac{\min_{1\le \kllevelrun \le \kllevel }\delta_{\kllevelrun}}{8 }\right\}&\le& 
\prob\left\{
\left\|\vec{\Sigma}^{(h;\nsamples)}_{\vec{\lnrfrv}^{(h)}}-\expect{\vec{\Sigma}^{(h;\nsamples)}_{\vec{\lnrfrv}^{(h)}}}\right\|_{2\to 2} > \frac{\bar{t}}{2} \right\} \nonumber \\ &
\le& 1-p_0 
\end{eqnarray}
using $36*8^2=9*2^8$ and thus $\sqrt{36*8^2}=3*2^4=48$.
This proves the statement.
\end{proof}

This establishes another key ingredient for our later analysis, namely the optimal bound on the sampling error.

\section{Error analysis for covariance operator reconstruction}\label{sec:covoprec}
We are now in a position to derive an overall  error bound for the reconstruction of the continuous covariance operator from finite samples. To do so, recall the 
covariance kernel  $R$ in \eqref{carlemannkern}.
Furthermore, let $\phi_{\kllevelrun}^{(h;\nsamples)}$ be given as in \eqref{mercerh3} and let
$\kllevel\in\mathbb{N}$ be a truncation parameter with $1\le \kllevel \le \dimVh$. 
We are now interested in the approximation error $R-R^{(\kllevel;h;\nsamples)}$, where $R^{(\kllevel;h;\nsamples)}$ was defined in \eqref{mercerh3}.
To this end, we define 
\begin{align}\label{deltagl}
G^2(\kllevel):=\sum_{\kllevelrun=1}^{\kllevel}\left( \delta^{-1}_{\kllevelrun}\lambda_{\kllevelrun}\right)^{2} ,
\end{align}
which is a function of the truncation parameter $\kllevel$ that is determined only by the spectral properties of the unknown covariance kernel  $R$. \\
In special cases, it is possible to compute $G(\kllevel)$ explicitly.
Consider Brownian motion as an example. In the even univariate case, we have the kernel
\begin{align}\label{univariateBrownianMotion}
	R^{(1)}_{B}:(0,1)^2 \to \R, \quad R^{(1)}_{B}(x,x^{\prime}):=\min\{x,x^{\prime}\}.
\end{align}
Then it is known that 
	$\lambda_{R^{(1)}_{B}}(\kllevelrun)= \pi^{-2}\left(\kllevelrun -\frac{1}{2} \right)^{-2}= 
	\left(\frac{2}{\pi}\right)^2 \left(2\ell-1 \right)^{-2} \in \left(\frac{2}{\pi}\right)^2 (2\N-1)^{-2}$.
So we have
$
\delta_{R^{(1)}_{B}}(\kllevelrun):= \frac{2\kllevelrun}{\pi^2(\kllevelrun^2-\frac{1}{4})^2}=\frac{1}{2}\left(\frac{2}{\pi}\right)^2 \frac{\kllevelrun}{\left(\kllevelrun^2 -\frac{1}{4}\right)^{2}}
$
and consequently 
we finally get
\begin{align*}
	G^2_{R^{(1)}_{B}}(\kllevel)=\sum_{\kllevelrun=1}^{\kllevel}\left(\frac{ \lambda_{R^{(1)}_{B}}(\kllevelrun)} {\delta_{R^{(1)}_{B}}(\kllevelrun)} \right)^{2} = \frac{1}{64}\sum_{\kllevelrun=1}^{\kllevel} \frac{(2\kllevelrun+1)^4}{\kllevelrun^2}\approx \kllevel^{3} \quad \text{for } \kllevel \text{ large.}
\end{align*}
Furthermore, for a two-dimensional Brownian motion, we have the eigenvalues
\begin{align*}
\lambda_{R^{(2)}_{B}}(\ell_1,\ell_2):= \lambda_{R^{(1)}_{B}}(\kllevelrun_{1}) \lambda_{R^{(1)}_{B}}(\kllevelrun_{2})=\left(\frac{2}{\pi} \right)^{4} (2\ell_1-1)^{-2}(2\ell_2-1)^{-2}, \quad \ell_1,\ell_2 \in \N.
\end{align*}
So we observe 
$\text{image}(\lambda_{R^{(2)}_{B}})=\lambda_{R^{(2)}_{B}}(\N \times \N) = \left(\frac{2}{\pi} \right)^{4} (2\N-1)^{-2}$,
i.e. the image of $\lambda_{R^{(2)}_{B}}$ up to the pre-factor $\left(\frac{2}{\pi} \right)^{4}$ again consists of the squared inverses of the odd numbers. 
Thus, we can sort the eigenvalues (ignoring multiplicities)  by the ordering in the odd natural numbers. Furthermore, it holds that
\begin{align}
\lambda_{R^{(2)}_{B}}(\ell_1,\ell_2)=\lambda_{R^{(1)}_{B}}(1) \lambda_{R^{(1)}_{B}}\left(k(\ell_1,\ell_2)\right), \quad \text{where } k(\ell_1,\ell_2):=\frac{(2\ell_1-1)(2\ell_2-1)+1}{2} .
\end{align}
In the general $d$-dimensional situation, we have 
\begin{align*}
	\lambda_{R^{(d)}_{B}}(\ell_1,\dots,\ell_d)=\prod_{j=1}^{d} \lambda_{R^{(d)}_{B}}(\ell_{j})=\left(\lambda_{R^{(1)}_{B}}(1) \right)^{d-1} \lambda_{R^{(1)}_{B}}(k(\ell_1,\dots,\ell_{d})),
\end{align*}
where 
$k(\ell_1,\dots,\ell_{d}):=\frac{1}{2} \left(\prod_{j=1}^{d} (2\ell_j-1) +1\right)$.
Furthermore, with $\ell:=k(\ell_1,\dots,\ell_{d})$ we have the bound
$\delta_{R^{(d)}_{B}}(\kllevelrun)=\left(\lambda_{R^{(1)}_{B}}(1) \right)^{d-1}\delta_{R^{(1)}_{B}}(\kllevelrun)$.
Thus, for the $d$-dimensional Brownian motion, we get for large $L$
\begin{align}\label{eq:gl:asymptotic}
G^2_{R^{(d)}_{B}}(\kllevel)=\sum_{\kllevelrun=1}^{\kllevel}\left(\frac{ \lambda_{R^{(d)}_{B}}(\kllevelrun)} {\delta_{R^{(d)}_{B}}(\kllevelrun)} \right)^{2} = \sum_{\kllevelrun=1}^{\kllevel} \left( \frac{\ \lambda_{R^{(1)}_{B}}(\kllevelrun)}{\delta_{R^{(1)}_{B}}(\kllevelrun)} \right)^{2}=\frac{1}{64}\sum_{\kllevelrun=1}^{\kllevel} \frac{(2\kllevelrun+1)^4}{\kllevelrun^2}\approx \kllevel^{3}.
\end{align}
Note that $G_{R^{d}_{B}}(\kllevel)$ is now completely independent of the dimension.
In the general situation we have the following result: 
\begin{theorem}\label{finalthm}
Let $\dimVh \ge \kllevel$.
Under the condition of \cref{prop:spectralgap}, i.e. \eqref{condspectralgapfinal1} and the assumption $C_2 h^{s}  \lambda_{\kllevel}^{-1} \le  1$,  it holds with probability $p_0$ given in \eqref{eq:ppp} that 
\begin{align}\label{eq:final-1}
\norm{{R}-R^{(\kllevel;h;\nsamples)}}_{L^{2}(\dom\times\dom)}&\lesssim  \kllevel^{-\frac{2s}{d}-\frac{1}{2}}  + \kllevel^{\frac{1}{2}} h^{s}  +    (\kllevel^{\frac{1}{2}} + G(\kllevel)) \mathcal{E}_{(h;M)}
\end{align}
with $G(\kllevel)$ from \eqref{deltagl}. If we use the assumption $C_2 h^{s} \lambda_{\kllevel}^{-1} \le 1$ again, we get the simplified bound  
\begin{align}\label{eq:final-2}
\norm{{R}-R^{(\kllevel;h;\nsamples)}}_{L^{2}(\dom\times\dom)}&\lesssim \kllevel^{-\frac{2s}{d}-\frac{1}{2}}+  (\kllevel^{\frac{1}{2}}  +G(\kllevel))\mathcal{E}_{(h;M)} .
\end{align}
\end{theorem}
\begin{proof}
Since \eqref{condspectralgapfinal1} holds, it is ensured that the condition \eqref{condspectralgap1} in our spectral gap \cref{ass:spectralgap} is satisfied with the probability $p_0$ from \eqref{eq:ppp}.
Now recall the definition of $R$ in \eqref{carlemannkern}, of $R^{(\kllevel)}$ in \eqref{carlemannkern}, of $R^{(\kllevel;h)}$ in \eqref{mercerh} and of $R^{(\kllevel;h;\nsamples)}$ in \eqref{mercerh3}.
We will break the proof into three parts. Note that we can decompose the approximation error as
\begin{align}\label{errordecomposition}
&R(\vec{x},\vec{x}^{\prime})-R^{(\kllevel;h;\nsamples)}(\vec{x},\vec{x}^{\prime}) := \underbrace{R(\vec{x},\vec{x}^{\prime})- R^{(\kllevel)}(\vec{x},\vec{x}^{\prime})}_{E_1}\\&
+\underbrace{R^{(\kllevel)}(\vec{x},\vec{x}^{\prime})-R^{(\kllevel;h)}(\vec{x},\vec{x}^{\prime})}_{E_2}
+\underbrace{R^{(\kllevel;h)}(\vec{x},\vec{x}^{\prime})
-R^{(\kllevel;h;\nsamples)}(\vec{x},\vec{x}^{\prime})}_{E_3} \nonumber.
\end{align}
The first error term
\begin{align}\label{error1}
E_1(\vec{x},\vec{x}^{\prime}):=R(\vec{x},\vec{x}^{\prime})- R^{(\kllevel)}(\vec{x},\vec{x}^{\prime})=\sum_{\kllevelrun=\kllevel+1}^{\infty}\lambda_{\kllevelrun}\phi_{\kllevelrun}(\vec{x})\phi_{\kllevelrun}(\vec{x}^{\prime})
\end{align}
depends only on the truncation parameter $\kllevel$.
The second error term
\begin{align}\label{error2}
E_2(\vec{x},\vec{x}^{\prime}):=R^{(\kllevel)}(\vec{x},\vec{x}^{\prime})-R^{(\kllevel;h)}(\vec{x},\vec{x}^{\prime})
\end{align}
depends on the spatial discretization $h$ and the truncation parameter $L$.
The third error term
\begin{align}\label{error3}
E_3(\vec{x},\vec{x}^{\prime})&:=R^{(\kllevel;h)}(\vec{x},\vec{x}^{\prime})-R^{(\kllevel;h;\nsamples)}(\vec{x},\vec{x}^{\prime})\nonumber \\&=\sum_{\kllevelrun=1}^{\kllevel} \lambda^{(h)}_{\kllevelrun} \phi^{(h)}_{\kllevelrun}(\vec{x}) \phi^{(h)}_{\kllevelrun}(\vec{x}^{\prime})-\sum_{\kllevelrun=1}^{\kllevel} \lambda^{(h;\nsamples)}_{\kllevelrun} \phi^{(h;\nsamples)}_{\kllevelrun}(\vec{x})\phi^{(h;\nsamples)}_{\kllevelrun}(\vec{x}^{\prime})
\end{align}
depends on the sample size $\nsamples$, the spatial discretization $h$, and the truncation parameter $L$.

To derive a bound for \eqref{error1} we use the orthonormality of the eigenfunctions $\{\phi_{\ell}\}_{\ell=1}^{\infty}$ and obtain directly for the integral operator
$
\mathcal{E}_{1}:L^{2}(\dom) \to L^{2}(\dom), \quad v \mapsto \int_{\dom}E_{1}(\vec{x},\cdot) v(\vec{x}) \, \mathrm{d}\lebesgue(\vec{x}) 
$
the identity
\begin{align}\label{E1error}
\norm{\mathcal{E}_{1}}_{L^{2}(\dom)\to L^{2}(\dom) } = \norm{E_1}_{L^{2}(\dom \times \dom)}  = \left(\sum_{\kllevelrun=\kllevel+1}^{\infty}  \lambda_{\kllevelrun}^2 \right)^{\frac{1}{2}}.
\end{align}
Now we can use the inequality \eqref{thm:truncationErrorLambda} of 
\cref{thm:truncationError} and derive the bound
\begin{align*}
\left(\sum_{\kllevelrun=\kllevel+1}^{\infty}  \lambda_{\kllevelrun}^2 \right)^{\frac{1}{2}} &\le C_{\ref{thm:truncationError}} 
\left(\sum_{\kllevelrun=\kllevel+1}^{\infty}  \kllevelrun^{-\frac{4s}{d}-2} \right)^{\frac{1}{2}} 
\le C_{\ref{thm:truncationError}} \left(\frac{4s}{d}+1 \right)^{-\frac{1}{2}} \kllevel^{-\frac{2s}{d}-\frac{1}{2}} .
\end{align*}
For the second term \eqref{error2} we have
\begin{align*}
&E_2(\vec{x},\vec{x}^{\prime}):=\sum_{\kllevelrun=1}^{\kllevel} \lambda_{\kllevelrun} \phi_{\kllevelrun}(\vec{x})\phi_{\kllevelrun}(\vec{x}^{\prime})-
\sum_{\kllevelrun=1}^{\kllevel} \lambda^{(h)}_{\kllevelrun} \phi^{(h)}_{\kllevelrun}(\vec{x})\phi^{(h)}_{\kllevelrun}(\vec{x}^{\prime})\\
&= \sum_{\kllevelrun=1}^{\kllevel} \left( \lambda_{\kllevelrun} -\lambda^{(h)}_{\kllevelrun} \right)\phi_{\kllevelrun}(\vec{x})\phi_{\kllevelrun}(\vec{x}^{\prime})
+\sum_{\kllevelrun=1}^{\kllevel} \lambda^{(h)}_{\kllevelrun} \left(\phi_{\kllevelrun}(\vec{x})-\phi^{(h)}_{\kllevelrun}(\vec{x})\right)
\phi_{\kllevelrun}(\vec{x}^{\prime})\\
&+\sum_{\kllevelrun=1}^{\kllevel} \lambda^{(h)}_{\kllevelrun} \phi^{(h)}_{\kllevelrun}(\vec{x}) \left(\phi_{\kllevelrun}(\vec{x}^{\prime})-\phi^{(h)}_{\kllevelrun}(\vec{x}^{\prime})\right)
=:E_{2;1}(\vec{x},\vec{x}^{\prime})+E_{2;2}(\vec{x},\vec{x}^{\prime})+E_{2;3}(\vec{x},\vec{x}^{\prime}).
\end{align*}
The orthonormality of $\{\phi_{\ell}\}_{\ell=1}^{\infty}$ and an application of 
\cref{prop:babuska} leads to  
\begin{align}\label{eq:yyy1}
\normL{{E}_{2;1}}{\dom\times\dom}&\le \left(\sum_{\kllevelrun=1}^{\kllevel} |\lambda_{\kllevelrun} -\lambda^{(h)}_{\kllevelrun}|^2\right)^{1/2}
\le C_1 h^{s}  \left(\sum_{\kllevelrun=1}^{\kllevel} h^{2s}\lambda_{\kllevelrun}^{-2}\right)^{1/2} \le \frac{C_1}{C_2} \kllevel^{\frac{1}{2}}h^s,
\end{align}
where the last inequality holds because $C_2 h^{s}  \lambda_{\kllevel}^{-1} \le 1$. 
This bound implies a condition on $h$, namely
\begin{align}\label{condhL}
C_2 h^{s} \le  \lambda_{\kllevel} \le \kllevel^{-\frac{2s}{d}-1},
\end{align}
where the first inequality is true in general, see \cref{lem:regeigenfunctions} and the upper bound follows only under additional assumptions.
Now we turn to $E_{2;2}$. The orthonormality of $\{\phi_{\ell}\}_{\ell=1}^{\infty}$ and the bound $\lambda^{(h)}_{\kllevelrun} \le \lambda_{\kllevelrun}$ for $1\le \kllevelrun \le \kllevel$  implies,
due to Courant's min-max principle for decreasing eigenvalues, the bound
\begin{align*}
\normL{{E}_{2;2}}{\dom\times\dom}^2
\le\sum_{\kllevelrun=1}^{\kllevel}\left(\lambda^{(h)}_{\kllevelrun}\right)^{2}
\normL{\phi_{\kllevelrun}-\phi_{\kllevelrun}^{(h)}}{\dom}^2 \le \sum_{\kllevelrun=1}^{\kllevel}\lambda^{2}_{\kllevelrun}
\normL{\phi_{\kllevelrun}-\phi_{\kllevelrun}^{(h)}}{\dom}^2 \le C_2^{2} h^{2s} \kllevel,
\end{align*}
where we used  $\normL{\phi_{\kllevelrun}^{(h)}-\phi_{\kllevelrun}}{\dom}\leq C_{2}\lambda_{\kllevelrun}^{-1} h^{s}$  in the last step, see \cref{prop:babuska}.
So we get
\begin{align}\label{eq:yyy2}
\normL{{E}_{2;2}}{\dom\times\dom} &\le C_{2} h^{s} \kllevel^{\frac{1}{2}}.
\end{align}
For the term $E_{2,3}$, the orthonormality of the series $\{\phi_{\kllevelrun}^{(h)}\}_{\kllevelrun=1}^{\dimVh}$, which follows with the eigenvalue problem 
\eqref{genalizedeigenvalue}, leads to
$\normL{{E}_{2;3}}{\dom\times\dom}^2
=\sum_{\kllevelrun=1}^{\kllevel}\left(\lambda^{(h)}_{\kllevelrun}\right)^2 \normL{\phi_{\kllevelrun}-\phi_{\kllevelrun}^{(h)}}{\dom}^2$.
So with literally the same computations as for \eqref{eq:yyy2}, we get the bound
\begin{align}\label{eq3}
\normL{{E}_{2;3}}{\dom\times\dom} &\le C_{2} h^{s} \kllevel^{\frac{1}{2}}.
\end{align}
Together with the estimates \eqref{eq:yyy1}, \eqref{eq:yyy2}, and \eqref{eq3}, we get
\begin{align}\label{E2error}
\normL{{E}_{2}}{\dom\times\dom}=\normL{{E}_{2;1}+{E}_{2;2}+{E}_{2;3}}{\dom\times\dom}
\le \left( \frac{C_1}{C_2} + 2 C_2 \right) C^{-1}_2\kllevel^{-\frac{2s}{d}-\frac{1}{2}}.
\end{align}
For the third term $E_3$ in \eqref{error3}, we proceed analogously to the splitting of \eqref{error2}. Then we use the decomposition
\begin{align*}
&E_3(\vec{x},\vec{x}^{\prime})=\sum_{\kllevelrun=1}^{\kllevel} \lambda^{(h)}_{\kllevelrun} \phi_{\kllevelrun}^{(h)}(\vec{x})\phi_{\kllevelrun}^{(h)}(\vec{x}^{\prime})
-\sum_{\kllevelrun=1}^{\kllevel} \lambda^{(h;\nsamples)}_{\kllevelrun} \phi^{(h;\nsamples)}_{\kllevelrun}(\vec{x})\phi^{(h;\nsamples)}_{\kllevelrun}(\vec{x}^{\prime})
\\&= \sum_{\kllevelrun=1}^{\kllevel} \left( \lambda^{(h)}_{\kllevelrun} -\lambda^{(h;\nsamples)}_{\kllevelrun} \right)
\phi_{\kllevelrun}^{(h)}(\vec{x})\phi_{\kllevelrun}^{(h)}(\vec{x}^{\prime})
+\sum_{\kllevelrun=1}^{\kllevel} \lambda^{(h;\nsamples)}_{\kllevelrun} \left(\phi_{\kllevelrun}^{(h)}(\vec{x})-\phi^{(h;\nsamples)}_{\kllevelrun}(\vec{x})\right)\phi^{(h)}_{\kllevelrun}(\vec{x}^{\prime}) \\&+\sum_{\kllevelrun=1}^{\kllevel} \lambda^{(h;\nsamples)}_{\kllevelrun} \left(\phi_{\kllevelrun}^{(h)}(\vec{x}^{\prime}) -\phi_{\kllevelrun}^{(h;\nsamples)}(\vec{x}^{\prime})\right)
\phi^{(h;\nsamples)}_{\kllevelrun}(\vec{x}) =:E_{3;1}+E_{3;2}+E_{3;3}.
\end{align*}
Now for the term $E_{3,1}$, the orthonormality of the basis $\{\phi_{\kllevelrun}^{(h)}\}_{\kllevelrun=1}^{\dimVh}$ and an application of \eqref{weyleigenvalue} implies
\begin{align}\label{eq:zzz1a}
\normL{{E}_{3;1}}{\dom}
&=\left( \sum_{\kllevelrun=1}^{\kllevel} \left( \lambda^{(h)}_{\kllevelrun} -\lambda^{(h;\nsamples)}_{\kllevelrun} \right)^2\right)^{\frac{1}{2}} \le \kllevel^{1/2} \mathcal{E}_{(h;M)}.
\end{align}
\noindent
To derive an upper estimate for the term $E_{3,2}$, we will use our bound from 
\cref{theoremeigenfucntions} for the eigenfunction approximation in \eqref{phivecerror}. 
So we have with probability $p_0$ that
\begin{align}\label{phierror}
\normL{\phi^{(h)}_{\kllevelrun}-\phi^{(h;\nsamples)}_{\kllevelrun}}{\dom}^2\le 4C^2  \delta^{-2}_{\kllevelrun}
\mathcal{E}_{(h;M)},
\end{align}
where we note at this point that this bound uses \eqref{condspectralgap1} from the spectral gap 
\cref{ass:spectralgap}.
 Due to the orthonormality of the basis $\{ \phi^{(h)}_{\kllevelrun}\}_{\kllevelrun=1}^{\dimVh}$, this results in 
\begin{align*}
\normL{{E}_{3;2}}{\dom\times\dom}
\le 2C\left\|\tildstiffnessmatrix-\tildestiffnessmatrixM \right\|_{2\to 2} \left(\sum_{\kllevelrun=1}^{\kllevel}\left(  \delta^{-1}_{\kllevelrun}\lambda^{(h;\nsamples)}_{\kllevelrun}
 \right)^{2}\right)^{\frac{1}{2}}.
\end{align*}
We can also use \eqref{weyleigenvalue}, i.e.
$
\left| \lambda^{(h)}_{\kllevelrun} -\lambda^{(h;\nsamples)}_{\kllevelrun} \right| \le 
\mathcal{E}_{(h;M)}
$, 
for $1\le  \kllevelrun \le \min\{\dimVh,\kllevel\}$
to get
$
\left|\lambda^{(h;\nsamples)}_{\kllevelrun} \right|\le \left| \lambda^{(h)}_{\kllevelrun} -\lambda^{(h;\nsamples)}_{\kllevelrun} \right| +\left| \lambda^{(h)}_{\kllevelrun} \right| \le 
 \mathcal{E}_{(h;M)}+\left| \lambda_{\kllevelrun} \right|,
$
where we again used $\lambda^{(h)}_{\kllevelrun} \le \lambda_{\kllevelrun}$ for $1\le \kllevelrun \le \kllevel$ in the last step.
This gives with probability $p_0$
\begin{align}\label{eq:zzz2}
\normL{{E}_{3;2}}{\dom\times\dom}
\le 2C  \mathcal{E}_{(h;M)} \left(\sum_{\kllevelrun=1}^{\kllevel}\left(  \delta^{-1}_{\kllevelrun} 
\mathcal{E}_{(h;M)} +  \delta^{-1}_{\kllevelrun} \left| \lambda_{\kllevelrun} \right| \right)^{2}\right)^{\frac{1}{2}}.
\end{align}
For the term $E_{3,3}$ we use the orthonormality of the $\{\phi_{\kllevelrun}^{(h;\nsamples)}\}$ to get
\begin{align}\label{e32e33}
\normL{E_{3;3}}{\dom\times \dom}^2&=\sum_{\kllevelrun=1}^{\kllevel} (\lambda^{(h;\nsamples)}_{\kllevelrun})^2 \normL{\phi^{(h)}_{\kllevelrun}-\phi^{(h;\nsamples)}_{\kllevelrun}}{\dom}^2 =\normL{E_{3;2}}{\dom\times \dom}^2.
\end{align}
So again we get with probability $p_0$
\begin{align}\label{E33error}
\normL{{E}_{3;3}}{\dom\times\dom}
\le 2C  \mathcal{E}_{(h;M)} \left(\sum_{\kllevelrun=1}^{\kllevel}\left(  \delta^{-1}_{\kllevelrun}\mathcal{E}_{(h;M)}  +  \delta^{-1}_{\kllevelrun} \left| \lambda_{\kllevelrun} \right| \right)^{2}\right)^{\frac{1}{2}}.
\end{align}

Combining the inequalities \eqref{eq:zzz1a}, \eqref{eq:zzz2}, \eqref{E33error} and using the triangle inequality $$\normL{{E}_{3}}{\dom\times\dom} \le \normL{{E}_{3;1}}{\dom\times\dom}+\normL{{E}_{3;2}}{\dom\times\dom}+\normL{{E}_{3;3}}{\dom\times\dom}$$  
we get again with probability $p_0$
\begin{align}\label{E3error1}
&\normL{{E}_{3}}{\dom\times\dom}
\le
 \mathcal{E}_{(h;M)} \left( \kllevel^{\frac{1}{2}}+4C \left(\sum_{\kllevelrun=1}^{\kllevel}\left( \delta^{-1}_{\kllevelrun}
 \mathcal{E}_{(h;M)}  +  \delta^{-1}_{\kllevelrun}\left| \lambda_{\kllevelrun} \right| \right)^{2}\right)^{\frac{1}{2}} \right).
\end{align}
So we get with  \eqref{condspectralgap1} from the spectral gap 
\cref{ass:spectralgap}
$\delta_{\kllevelrun}
\ge 4\mathcal{E}_{(h;M)},
$
which is satisfied  with probability $p_0$, that 
\begin{align*}
&\sum_{\kllevelrun=1}^{\kllevel}\left(  \delta^{-1}_{\kllevelrun}\mathcal{E}_{(h;M)}  +  
 \delta^{-1}_{\kllevelrun}\left| \lambda_{\kllevelrun} \right| \right)^{2} \le 
2 \sum_{\kllevelrun=1}^{\kllevel}\left(  \delta^{-1}_{\kllevelrun} \mathcal{E}_{(h;M)}\right)^{2} +2
 \sum_{\kllevelrun=1}^{\kllevel} \left( \delta^{-1}_{\kllevelrun} \left| \lambda_{\kllevelrun} \right| \right)^{2}
\nonumber \\&\le\frac{1}{8}  \kllevel +  2 \sum_{\kllevelrun=1}^{\kllevel} \left( \delta^{-1}_{\kllevelrun}\lambda_{\kllevelrun}  \right)^{2}
= \frac{1}{8}  \kllevel + 2 G^2(\kllevel).
\end{align*}
Thus,  using $\sqrt{a+b} \le \sqrt{a}+\sqrt{b}$ for $a,b\ge 0$, the bound \eqref{E3error1} reduces to
\begin{align}\label{E3error}
\normL{{E}_{3}}{\dom\times\dom} \le
 \mathcal{E}_{(h;M)} \left( \kllevel^{\frac{1}{2}}+\sqrt{2}C \kllevel^{\frac{1}{2}} + 4\sqrt{2} C G(\kllevel)\right),
\end{align}
which holds with  probability $p_0$.
Finally, combining the inequalities \eqref{E1error}, \eqref{E2error} and \eqref{E3error} we get 
\begin{align*}
\norm{R-R^{(\kllevel;h;\nsamples)}}_{\dom \times \dom} \le \norm{E_1 + E_2 + E_3}_{\dom \times \dom} \le \norm{E_1 }_{\dom \times \dom} +\norm{E_2 }_{\dom \times \dom} +\norm{E_3}_{\dom \times \dom}, 
\end{align*}
which shows the assertion.
\end{proof}
The bound in \cref{finalthm} still depends on $\mathcal{E}_{(h;M)}$. 
However, this term includes the $\nsamples$-dependent sampling approximation \eqref{tildestiffnessample} and is thus a random variable. In the following, we are interested in the expected value of the error. 
Note that we cannot use \cref{corboundtildeS} directly, since we need the spectral gap \cref{ass:spectralgap} to be valid. But this is only 
the case with probability $p_0$, which needs to be taken care of properly.
To do this, we introduce the quantity
\begin{align}\label{HL}
H(\kllevel):= {\left(\frac{1}{48}\min_{1\le \kllevelrun \le \kllevel }\delta_{\kllevelrun}\right)^{2}},
\end{align}
which is a function of the truncation parameter $\kllevel$ that depends on the spectral properties as given in \eqref{defn:gap} of the true kernel $R$. For general operators, this quantity is not easy to obtain.

However, in the simple case of $d$-dimensional Brownian motion, we can obtain a precise value as
\begin{align}\label{eq:Hlasymptotic}
H_{R^{(d)}_{B}}(\kllevel)&:=\frac{1}{{48^{2}}}\min_{\kllevelrun=1}^{\kllevel} \delta^{2}_{R^{(d)}_{B}}(\kllevelrun)=\frac{1}{{2304}} \min_{\kllevelrun=1}^{\kllevel}\left(\lambda_{R^{(1)}_{B}}(1) \right)^{2(d-1)}\left(\delta_{R^{(1)}_{B}}(\kllevelrun)\right)^{2}\nonumber \\
&=\frac{1}{{2304}} \left(\lambda_{R^{(1)}_{B}}(1) \right)^{2(d-1)} \frac{1}{4}\left(\frac{2}{\pi}\right)^4\min_{\kllevelrun=1}^{\kllevel}\left(\frac{\kllevelrun}{\left(\kllevelrun^2 -\frac{1}{4}\right)^{2}}\right)^{2} 
\approx \kllevel^{-6} \quad \text{for } \kllevel \text{ large.}
\end{align}
Note that unlike the value $G_{R^{(d)}_{B}}(\kllevel)$ in \eqref{eq:gl:asymptotic}, the dimension now enters exponentially into the total value of $H_{R^{(d)}_{B}}(\kllevel)$ in \eqref{eq:Hlasymptotic} .

All in all, in the general situation we have the following result:
\begin{theorem}\label{thm:final2}
Under the conditions of \cref{prop:spectralgap}, i.e. \eqref{condspectralgapfinal1}, with the condition \eqref{condhL} and with the definition \eqref{tilderhoM}, we have
\begin{align}\label{eq:final-3}
\expect{\norm{{R}-R^{(\kllevel;h;\nsamples)}}_{L^{2}(\dom\times\dom)}} &\lesssim  \kllevel^{-\frac{2s}{d}-\frac{1}{2}} + \left( \kllevel^{\frac{1}{2}}+ G(\kllevel) \right) \tilde{\rho}^{\frac{1}{2}}_{h}(\nsamples) \lambda_{\max}\left(\massmatrix\right) \nonumber\\
	&+  L^{\frac{1}{2}} h^{-d} \exp\left(- \nsamples  \rho_1 H(\kllevel)\lambda_{\max}^{-2} \left(\massmatrix\right)  \right),
\end{align}
where $H(\kllevel)$ is defined in \eqref{HL} and $\rho_1$ is a constant given in the proof of 
 \cref{prop:spectralgap}.
\end{theorem}
\begin{proof}
From 
\cref{finalthm}, we have with the condition \eqref{condhL} that
\begin{align*}
&\expect{\norm{{R}-R^{(\kllevel;h;\nsamples)}}_{L^{2}(\dom\times\dom)}} \lesssim  \kllevel^{-\frac{2s}{d}-\frac{1}{2}}   + h^{s} \kllevel^{\frac{1}{2}}+  \expect{ \normL{E_{3}}{\dom\times \dom}} \\
&\lesssim  \kllevel^{-\frac{2s}{d}-\frac{1}{2}}  +  \expect{ \normL{E_{3;1}}{\dom\times \dom}+ \normL{E_{3;2}}{\dom\times \dom}+ \normL{E_{3;3}}{\dom\times \dom}}. 
\end{align*}
The term $\normL{{E}_{3;1}}{\dom\times \dom}$ from \eqref{eq:zzz1a} can be handled directly and we get with \cref{corboundtildeS} the bound
\begin{align}\label{expectE31}
\expect{\normL{{E}_{3;1}}{\dom\times \dom}}&\le  \kllevel^{1/2} \expect{\mathcal{E}_{(h;M)}} \lesssim  
\kllevel^{1/2}  \rho^{\frac{1}{2}}_{h}(\nsamples) \lambda_{\max}\left(\massmatrix\right) .
\end{align}
Next, we define the set in which the \cref{ass:spectralgap} is violated as bad, and the good set simply as the complement of the bad set, i.e.
\begin{align*}
M_{\text{bad}}:=\left\{\!\vec{\om} \in \Omega :  \min_{1\le \kllevelrun \le \kllevel }\delta_{\kllevelrun}< 4
\mathcal{E}_{(h;M)}\!\right\} \text{ and }
M_{\text{good}}:=\left\{\! \vec{\om} \in \Omega :  \min_{1\le \kllevelrun \le \kllevel }\delta_{\kllevelrun}\ge 4
 \mathcal{E}_{(h;M)}\! \right\}.
\end{align*}
We notice that \cref{prop:spectralgap} yields 
$\left| M_{\text{bad}}\right| \le 2 \dimVh 5^{\taperingint}\exp\left(- \nsamples  \rho_1   H(\kllevel) \lambda_{\max}^{-2} \left(\massmatrix\right) \right),
$
where $\tau =\nsamples^{\frac{1}{2\alpha+1}}$.
Now we split the term $\expect{\normL{{E}_{3;j}}{\dom\times \dom}}$ as 
\begin{align*}
\expect{\normL{{E}_{3;j}}{\dom\times \dom}}=\int_{ M_{\text{bad}} } \normL{{E}_{3;j}}{\dom\times \dom} \, \mathrm{d}\prob(\vec{\om})+\int_{ M_{\text{good}} } \normL{{E}_{3;j}}{\dom\times \dom} \, \mathrm{d}\prob(\vec{\om}),
\end{align*}
and we also remember that $\normL{{E}_{3;2}}{\dom\times \dom}=\normL{{E}_{3;3}}{\dom\times \dom}$ as we saw in \eqref{e32e33}.
To treat the first integral concerning the bad set, we will show that the integrand is pointwise bounded by a deterministic quantity on $ M_{\text{bad}}$, i.e. that $\normL{{E}_{3;j}}{\dom\times \dom} \le C$ holds.
First we use \eqref{philhmortho} to get at least
$
\normL{\phi^{(h)}_{\kllevelrun}-\phi^{(h;\nsamples)}_{\kllevelrun}}{\dom}^2\le 2\normL{\phi^{(h)}_{\kllevelrun}}{\dom}^2+2\normL{\phi^{(h;\nsamples)}_{\kllevelrun}}{\dom}^2\le 4.
$
We also use the bound
$$
\left|\lambda^{(h;\nsamples)}_{\kllevelrun} \right|\le \left| \lambda^{(h)}_{\kllevelrun} -\lambda^{(h;\nsamples)}_{\kllevelrun} \right| +\left| \lambda^{(h)}_{\kllevelrun} \right| \le \mathcal{E}_{(h;M)}+\left| \lambda_{\kllevelrun} \right|.
$$
So we get for $j=2,3$ the bound
\begin{align*}
\normL{{E}_{3;j}}{\dom\times\dom}^2
&=\sum_{\kllevelrun=1}^{\kllevel} (\lambda^{(h;\nsamples)}_{\kllevelrun})^2 \normL{\phi^{(h)}_{\kllevelrun}-\phi^{(h;\nsamples)}_{\kllevelrun}}{\dom}^2 \le 4 \sum_{\kllevelrun=1}^{\kllevel} (\lambda^{(h;\nsamples)}_{\kllevelrun})^2 
\\&\le 4 \sum_{\kllevelrun=1}^{\kllevel} \left(  \mathcal{E}_{(h;M)} +\left| \lambda_{\kllevelrun} \right|\right)^2  \le 8  \sum_{\kllevelrun=1}^{\kllevel}  \mathcal{E}^2_{(h;M)}  + 8 \sum_{\kllevelrun=1}^{\kllevel}\left| \lambda_{\kllevelrun} \right|^2
\\&\le 8  \kllevel  \mathcal{E}^2_{(h;M)}  + 8\kllevel \left| \lambda_{1} \right|^2
\end{align*}
and
we get
$
\normL{{E}_{3;j}}{\dom\times\dom}\le 2\sqrt{2} \ \kllevel^{\frac{1}{2}}\left( \mathcal{E}_{(h;M)} +\lambda_{1}  \right).
$
Hence, on the bad set we get
\begin{align*}
\int_{\vec{\om} \in M_{\text{bad}} } \normL{{E}_{3;j}}{\dom\times \dom} \, \mathrm{d}\prob(\vec{\om})\le  2\sqrt{2}  \ \kllevel^{\frac{1}{2}} \int_{\vec{\om} \in M_{\text{bad}} }  \mathcal{E}_{(h;M)} \, \mathrm{d}\prob(\vec{\om}) +  2\sqrt{2}  \ \kllevel^{\frac{1}{2}} \left|  M_{\text{bad}} \right|  \lambda_{1}.
\end{align*}
For the first term, we use \eqref{errorexpec} to get
\begin{align*}
\int_{\vec{\om} \in M_{\text{bad}} }   \mathcal{E}_{(h;M)} \, \mathrm{d}\prob(\vec{\om})& \le \int_{\vec{\om} \in \Omega }    \mathcal{E}_{(h;M)} \, \mathrm{d}\prob(\vec{\om})
\lesssim  \rho^{\frac{1}{2}}_{h}(\nsamples)\lambda_{\max}\left(\massmatrix\right) 
\end{align*}
and obtain 
\begin{align*}
\int_{\vec{\om} \in M_{\text{bad}} } \normL{{E}_{3;j}}{\dom\times \dom} \, \mathrm{d}\prob(\vec{\om})\lesssim 
2 \sqrt{2}L^{\frac{1}{2}} \lambda_{\max}\left(\massmatrix\right) 
\rho^{\frac{1}{2}}_{h}(\nsamples)+2\sqrt{2} \left|  M_{\text{bad}} \right| L^{\frac{1}{2}} \lambda_1.
\end{align*}
We are left with estimating the integral of $\normL{E_{3}}{\dom\times\dom}$ over the set $M_{\text{good}}$. 
For this we use the bound \eqref{E3error}, which is only valid on the good set $ M_{\text{good}}$. Now recall
\begin{align*}
\normL{{E}_{3}}{\dom\times\dom} \le
 \mathcal{E}_{(h;M)} \left( \kllevel^{\frac{1}{2}}+\sqrt{2}C \kllevel^{\frac{1}{2}} + 4\sqrt{2} C G(\kllevel)\right) \quad \text{on }M_{\text{good}}.
\end{align*}
Then we get 
\begin{align*}
\int_{M_{\text{good}} } \normL{E_{3}}{\dom\times\dom} \, \mathrm{d}\prob (\vec{\om}) & 
\le \left( \kllevel^{\frac{1}{2}}+\sqrt{2}C \kllevel^{\frac{1}{2}} + 4\sqrt{2} C G(\kllevel)\right) \int_{M_{\text{good}}}  
\mathcal{E}_{(h;M)}\, \mathrm{d}\prob (\vec{\om}) 
\\&\le \left( \kllevel^{\frac{1}{2}}+\sqrt{2}C \kllevel^{\frac{1}{2}} + 4\sqrt{2} C G(\kllevel)\right)\expect{ 
\mathcal{E}_{(h;M)} }
\\& \lesssim \left( \kllevel^{\frac{1}{2}}+\sqrt{2}C \kllevel^{\frac{1}{2}} + 4\sqrt{2} C G(\kllevel)\right) \rho^{\frac{1}{2}}_{h}(\nsamples)\lambda_{\max}\left(\massmatrix\right) .
\end{align*}
With \cref{prop:spectralgap} and the notation from \eqref{tilderhoM} this completes the proof.
\end{proof}

\section{Discussion}
It remains to put the result of \cref{thm:final2} into context. To do this, let us consider an error bound 
$$\expect{\| R-R^{(L;h;M)}\|_{L^{2}(\dom \times \dom )}} \leq c \varepsilon$$ 
with a prescribed, fixed accuracy $\varepsilon >0$ and a small constant $c$. The question then is: What are the conditions on the discretization parameters $L,h,M$ to achieve this goal? 
For simplicity, we restrict ourselves to the situation $\lambda_{\max}\left(\massmatrix\right) \approx \lambda_{\min}\left(\massmatrix\right) \approx 1$, which is satisfied, for example, for any orthonormal basis.\footnote{For the nodal basis we would obtain, after proper scaling,  $\lambda_{\max}\left(\massmatrix\right) \approx \lambda_{\min}\left(\massmatrix\right) \approx h^d$. The resulting discussion for this case is left to the reader.}
Furthermore, we assume that the smoothness assumption \eqref{A:11} holds with $s > d/2$.

To demonstrate the coupling, we consider a situation where $h$ can be reasonably small and $M$ is sufficiently large. This choice naturally affects the tapering estimator.
We set an accuracy $\varepsilon>0$ and consider the case $\dimVh > \nsamples^{\frac{1}{2\alpha+1}}$, $\nsamples^{-\frac{2\alpha}{2\alpha+1}} \le d \log(h^{-1})\nsamples^{-1}$. 
Then the following choices give the sufficient conditions
\begin{align*}
L_{\varepsilon} &= \lceil \varepsilon^{-\frac{2d}{4s+d}} \rceil,\\
\nsamples_{\varepsilon} &\gtrsim \max\left\{\tilde{\nsamples}_{\varepsilon},\lceil \varepsilon^{-\frac{2d}{4s+d}} \rceil^{\frac{2(2s+d)\beta+2sd\gamma}{sd} } \varepsilon^{-2} \right\},\\
 \left(\kllevel_{\varepsilon}^{\frac{1}{2}}\varepsilon^{-1}\exp\left(- \nsamples_{\varepsilon}  \rho_1 H(\kllevel_{\varepsilon}) \right)\right)^{\frac{1}{d}}&\lesssim  h_{\varepsilon} \lesssim  \min\left\{H^{\frac{1}{4s}}(\kllevel_{\varepsilon}) \lambda^{\frac{1}{2s}}_{\kllevel_{\varepsilon}+1},\lambda^{\frac{1}{s}}_{\kllevel_{\varepsilon}}, \nsamples_{\varepsilon}^{-\frac{1}{d(2\alpha+1)}} , h_0
\right\},
\end{align*}
to achieve a total error of size $\varepsilon$. This is formalized in \cref{cor:appendixcor} in the appendix.

In general, our estimates involve the values of $G(\kllevel)$ and $H(\kllevel)$, i.e. the spectral properties of the unknown kernel $R$. These values are not  easy to determine. However, in the case of Brownian motion, we were able to derive precise values in \eqref{eq:gl:asymptotic} and \eqref{eq:Hlasymptotic}. 
Finally, let us consider the simplest case of univariate Brownian motion and discuss the tapering estimator in this situation.
Precisely,  in the univariate Brownian motion situation, we get the following parameter choices:
\begin{align*}
    L_{\varepsilon} & \sim \lceil \varepsilon^{-\frac{2}{3}} \rceil \\
    \nsamples_{\varepsilon} &\gtrsim\max\left\{\lceil \frac{1}{\rho_1 H(\kllevel_{\varepsilon})} W \left( -\rho_1 H(\kllevel_{\varepsilon}) \left( \kllevel_{\varepsilon}^{-\frac{1}{2}}\varepsilon \right) \right) \rceil,\lceil \varepsilon^{-\frac{2}{3}} \rceil^{8\beta+3} \varepsilon^{-2} \right\}\\
    h_{\varepsilon}&\sim\min\left\{ \lceil \varepsilon^{\frac{10}{3}} \rceil , \nsamples_{\varepsilon}^{-1},h_{0} \right\},
\end{align*}
where $W$ is the product logarithm function. For a derivation, see \cref{cor:appendixcor}.

 Finally, we should note the cost of our resulting approximation algorithm with respect to $\kllevel, h,\nsamples$. The numerical cost does not depend directly on $\kllevel$, since $\kllevel$ is implicitly included in $h_{\varepsilon}$  as $\kllevel \le \dimVh \sim h^{-d}$. Furthermore, it involves $\mathcal{O}(\nsamples h^{-2d})$ operations to assemble the $\nsamples$  associated discrete eigensystem \eqref{genalizedeigenvalue3-estimator},  and it involves $\mathcal{O}(h^{-3d})$ operations to solve the eigensystem in a naive direct way. Note at this point that there are faster approximate solution techniques, such as multipole or algebraic multigrid methods, with substantially reduced cost to tackle the task of eigensystem solution.

A final analysis of the tradeoff between cost and accuracy, and thus the corresponding $\epsilon$-complexity of our approach, can be easily done based on our results above and is left to the reader.

\section{Concluding remarks}\label{sec:conclusion}
We have discussed the problem of approximating the covariance of a Gaussian random field from a finite number of discretized observations. To this end, we coupled recent sharp estimates of the eigenvalue decay of continuous covariance operators with optimal statistical (tapering) estimators for covariance matrices. The combination of these techniques included a finite element discretization, which made all operators finite-rank  and made our approach feasible.

We provided new and sharp error estimates in  expectation for the reconstruction of the full covariance operator, taking  into account the number of samples, the finite element discretization, and the truncation of the Karhunen-Lo\`{e}ve expansion of the covariance operator and thus its regularity.

\section*{Acknowledgments}
The author MG was supported in part by the Hausdorff Center for Mathematics in Bonn and the Sonderforschungsbereich 1060 {\em The Mathematics of Emergent Effects} funded by the Deutsche Forschungsgemeinschaft.
GL acknowledges support from the Royal Society through the Newton International Fellowship. Parts of a preliminary version of this work were obtained during her visit to IPAM in the Long Term Program: Computational Issues in Oil Field Applications.
CR thanks the Institute for Numerical Simulation for its hospitality.

\bibliographystyle{siamplain}


\section*{Appendix}
\begin{corollary}\label{cor:appendixcor}
We set an accuracy $\varepsilon>0$ and choose $L_{\varepsilon} = \lceil \varepsilon^{-\frac{2d}{4s+d}} \rceil $.
Then we select $\nsamples_{\varepsilon}$ according to \eqref{finalMeps},  \eqref{finalMeps2} and \eqref{finalheps3}. Next, we choose $h_{\varepsilon}$ as in \eqref{finalheps} or as the maximal value of the potential intervals in \eqref{finalheps2} and \eqref{finalheps3}.
Then, in all possible cases, the error bound 
$\expect{\norm{{R}-R^{(\kllevel;h;\nsamples)}}_{L^{2}(\dom\times\dom)}} \le 3 \varepsilon$
holds.
\end{corollary}
\begin{proof}
For the first term in our error bound \eqref{eq:final-3} we get from $L ^{-2s/d -1/2} \stackrel{!}{\le}  \varepsilon$ directly the condition $L \ge \varepsilon^{-2d/(4s+d)}$.
We also note that choosing a much larger $\kllevel$ will destroy the balance of the error contributions. Therefore, we will assume
\begin{align}\label{aaa}
L_{\varepsilon} := \lceil \varepsilon^{-\frac{2d}{4s+d}} \rceil.
\end{align}
The coupling \eqref{aaa} influences the choices for $h$ in terms of $\kllevel$ via the 
\cref{prop:spectralgap}, i.e. \eqref{condspectralgapfinal1} and the relation \eqref{condhL}. As a sufficient condition on $h$ we get
\begin{align}\label{finalcondhL}
h^{2s} \lesssim \min\left\{H^{\frac{1}{2}}(\kllevel) \lambda_{\kllevel+1},\lambda^2_{\kllevel} \right\} \lesssim  \min \left\{ \delta_{\kllevelrun} \lambda_{\kllevel+1},\lambda^{2}_{\kllevel}  \ : \ 1\le \kllevelrun \le \kllevel \right\}
\end{align}
with the notation \eqref{HL}. 
For the second error term we now assume that $G(\kllevel)=\kllevel^{\tilde{\gamma }}$
for some $\tilde{\gamma} \in \R$. So  
\begin{align}\label{defgamma}
\kllevel^{\frac{1}{2}}+G(\kllevel)\lesssim \kllevel^{\gamma}, \quad \text{with} \quad \gamma:=\max\{\frac{1}{2},\tilde{\gamma}\}.
\end{align}
Consequently, we have $(\kllevel^{\frac{1}{2}}+G(\kllevel)) \tilde{\rho}^{\frac{1}{2}}_{h}(\nsamples) \lesssim \kllevel^{\gamma} \tilde{\rho}^{\frac{1}{2}}_{h}(\nsamples) $. At this point we note that the clauses in the definition of $\tilde{\rho}_h(M)$ in \eqref{tilderhoM} require a case distinction.

In the case $\dimVh < \nsamples^{\frac{1}{2\alpha+1}}$ with  $\dimVh =s^{d} h^{-d}$, we get $h > s \nsamples^{-\frac{1}{d(2\alpha+1)}}$. So with \eqref{finalcondhL} we have the following inequality for $h$:
\begin{align}\label{firstchain}
\nsamples_{\varepsilon}^{-\frac{1}{d(2\alpha+1)}} \lesssim h_{\varepsilon} \lesssim  \min\left\{H^{\frac{1}{4s}}(\kllevel_{\varepsilon}) \lambda^{\frac{1}{2s}}_{\kllevel_{\varepsilon}+1},\lambda^{\frac{1}{s}}_{\kllevel_{\varepsilon}} \right\}.
\end{align}
We also get from \eqref{tilderhoM} the bound $\tilde{\rho}_h(M)\lesssim h^{-d} \nsamples^{-1}$. We therefore infer the inequality
	$(\kllevel_{\varepsilon}^{\frac{1}{2}}+G(\kllevel_{\varepsilon})) \tilde{\rho}^{\frac{1}{2}}_{h_{\varepsilon}}(\nsamples_{\varepsilon})\lesssim  \kllevel_{\varepsilon}^{\gamma} \nsamples_{\varepsilon}^{\frac {1}{2(2\alpha+1)}} \nsamples_{\varepsilon}^{-\frac{1}{2}} = 
	\nsamples_{\varepsilon}^{-\frac{\alpha}{2\alpha+1}}\kllevel_{\varepsilon}^{\gamma}$.
  This gives us the sufficient condition 
 \begin{align}\label{firstsuffcond}
\nsamples_{\varepsilon}^{-\frac{\alpha}{2\alpha+1}}\kllevel_{\varepsilon}^{\gamma}
	 \stackrel{!}{\le} \varepsilon 
\end{align}  
which implies an error bound of size $\varepsilon$. 
From this condition we want to derive a condition on $\nsamples_{\varepsilon}$.
The sufficient condition \eqref{firstsuffcond} implies 
$\nsamples_{\varepsilon}^{-\frac{\alpha}{2\alpha+1}}\kllevel_{\varepsilon}^{\gamma}
	 \le \varepsilon \Leftrightarrow  \nsamples_{\varepsilon} \ge \varepsilon^{-\frac{2\alpha+1}{\alpha}} \kllevel_{\varepsilon}^{\gamma\frac {2\alpha+1}{\alpha} }$,
i.e. we have a lower bound on $\nsamples_{\varepsilon}$. This means that we can satisfy \eqref{firstchain} by making  $\nsamples_{\varepsilon}$ sufficiently large.
Specifically, we have 
\begin{align}\label{condM1}
	\nsamples_{\varepsilon} \gtrsim \max\left\{ \varepsilon^{-\frac{2\alpha+1}{\alpha}} \kllevel_{\varepsilon}^{\gamma\frac {2\alpha+1}{\alpha} },\left( \min\left\{H^{\frac{1}{4s}}(\kllevel_{\varepsilon}) \lambda^{\frac{1}{2s}}_{\kllevel_{\varepsilon}+1},\lambda^{\frac{1}{s}}_{\kllevel_{\varepsilon}} \right\} \right)^{-d(2\alpha+1)}  \right\}.
\end{align}
Furthermore,  for the third error contribution we have the condition
\begin{align*}
 \kllevel_{\varepsilon}^{\frac{1}{2}} h_{\varepsilon}^{-d} \exp\left(- \nsamples_{\varepsilon}  \rho_1 H(\kllevel_{\varepsilon}) \right)\stackrel{!}{\le} \varepsilon \Leftrightarrow  \exp\left(- \nsamples_{\varepsilon}  \rho_1 H(\kllevel_{\varepsilon}) \right) \le  \varepsilon \kllevel_{\varepsilon}^{-\frac{1}{2}} h_{\varepsilon}^{d}.
 \end{align*}
We use \eqref{firstchain} and \eqref{aaa} to observe 
$\varepsilon \kllevel_{\varepsilon}^{-\frac{1}{2}} \nsamples_{\varepsilon}^{-\frac{1}{2\alpha+1}}=\varepsilon \lceil \varepsilon^{-\frac{2d}{4s+d}} \rceil^{-\frac{1}{2}} \nsamples_{\varepsilon}^{-\frac{1}{2\alpha+1}}	 \le \varepsilon \kllevel_{\varepsilon}^{-\frac{1}{2}} h_{\varepsilon}^{d}$.
Now we define $\bar{\nsamples}_{\varepsilon} \in \N$ by
\begin{align}\label{barn}
	\bar{\nsamples}_{\varepsilon}:=\arg\min\left\{M \in \N :  \exp\left(- M  \rho_1 H(\kllevel_{\varepsilon})\right) \le \varepsilon \lceil \varepsilon^{-\frac{2d}{4s+d}} \rceil^{-\frac{1}{2}} M^{-\frac{1}{2\alpha+1}} \right\}.
\end{align}
In summary, to ensure an overall error bound of size $3\varepsilon$ in the case $\dimVh < \nsamples^{\frac{1}{2\alpha+1}}$, this gives the sufficient conditions 
\begin{align}
\label{finalLeps}
L_{\varepsilon} &= \lceil \varepsilon^{-\frac{2d}{4s+d}} \rceil,\\
\label{finalMeps}
 \nsamples_{\varepsilon} &\gtrsim \max\left\{\bar{\nsamples}_{\varepsilon},  \varepsilon^{-\frac{2\alpha+1}{\alpha}} \kllevel_{\varepsilon}^{\gamma\frac {2\alpha+1}{\alpha} },\left( \min\left\{H^{\frac{1}{4s}}(\kllevel_{\varepsilon}) \lambda^{\frac{1}{2s}}_{\kllevel_{\varepsilon}+1},\lambda^{\frac{1}{s}}_{\kllevel_{\varepsilon}} \right\} \right)^{-d(2\alpha+1)}\right\},\\
\label{finalheps}
h_{\varepsilon} &\sim \min\left\{ \nsamples_{\varepsilon}^{-\frac{1}{d(2\alpha+1)}},h_0\right\},
\end{align}
where $h_0$ comes from \cref{prop:babuska}.

Now we turn to the second clause in the definition of $\tilde{\rho}_h(M)$ in \eqref{tilderhoM}, i.e. we consider the case $\dimVh \ge \nsamples^{\frac{1}{2\alpha+1}}$ with $\dimVh= s^{d}h^{-d}$. Then we get $h \le s  \nsamples^{-\frac{1}{d(2\alpha+1)}}$ and the inequalities \eqref{firstchain} change to 
\begin{align}\label{secondchain}
 h_{\varepsilon} \lesssim  \min\left\{H^{\frac{1}{4s}}(\kllevel_{\varepsilon}) \lambda^{\frac{1}{2s}}_{\kllevel_{\varepsilon}+1},\lambda^{\frac{1}{s}}_{\kllevel_{\varepsilon}}, \nsamples_{\varepsilon}^{-\frac{1}{d(2\alpha+1)}}  \right\},
\end{align} 
i.e. we have no lower bound on $h_{\varepsilon}$ at this point.
Recall \eqref{tilderhoM}, i.e. 
$\tilde{\rho}_{h}(\nsamples)=\nsamples^{-\frac{2\alpha}{2\alpha+1}} +d \log(h^{-1})\nsamples^{-1}$.
Now we distinguish whether $\nsamples^{-\frac{2\alpha}{2\alpha+1}} \le d \log(h^{-1})\nsamples^{-1}$ or  whether $\nsamples^{-\frac{2\alpha}{2\alpha+1}} \ge d \log(h^{-1})\nsamples^{-1}$ holds.
Furthermore, we observe that for all $ \beta\ge 0$ there exists a $h_{\beta}>0$ such that
 $\log(h^{-1})^{\frac{1}{2}} \le h^{-\beta}$ 
 for all  $h\le h_{\beta}$.
So for a fixed $\beta >0$ we get the chain of inequalities with $\tilde{\rho}_{h}(\nsamples_{\varepsilon})=\nsamples_{\varepsilon}^{-\frac{2\alpha}{2\alpha+1}} +d \log(h_{\varepsilon}^{-1})\nsamples_{\varepsilon}^{-1}\le 2 d \log(h_{\varepsilon}^{-1})\nsamples_{\varepsilon}^{-1}$.  Then we have 
\begin{align*}
(\kllevel_{\varepsilon}^{\frac{1}{2}}+G(\kllevel_{\varepsilon}))  \tilde{\rho}^{\frac{1}{2}}_{h_{\varepsilon}}(\nsamples_{\varepsilon})&\lesssim  \kllevel_{\varepsilon}^{\gamma}  d^{\frac{1}{2}} \log(h_{\varepsilon}^{-1})^{\frac{1}{2}}\nsamples^{-\frac{1}{2}} \lesssim \kllevel_{\varepsilon}^{\gamma} h_{\varepsilon}^{-\beta} \nsamples^{-\frac{1}{2}}
 \le  \kllevel_{\varepsilon}^{\gamma}  \nsamples_{\varepsilon}^{-\frac{1}{2}}\lambda_{\kllevel_{\varepsilon}}^{-\frac{\beta}{s}} \\&\le  \kllevel_{\varepsilon}^{\gamma}  \nsamples_{\varepsilon}^{-\frac{1}{2}} \kllevel_{\varepsilon}^{\frac{(2s+d)\beta}{ds}}= \nsamples_{\varepsilon}^{-\frac{1}{2}}\kllevel_{\varepsilon}^{\frac{(2s+d)\beta+sd\gamma}{ds} }.
  \end{align*}
Now, using \eqref{aaa}, we find the sufficient condition 
\begin{align*}
\nsamples_{\varepsilon}^{-\frac{1}{2}}\kllevel_{\varepsilon}^{\frac{(2s+d)\beta+sd\gamma}{sd} } \stackrel{!}{\le } \varepsilon \Leftrightarrow \nsamples_{\varepsilon} \ge \kllevel_{\varepsilon}^{\frac{2(2s+d)\beta+2sd\gamma}{sd} } \varepsilon^{-2}=\lceil \varepsilon^{-\frac{2d}{4s+d}} \rceil^{\frac{2(2s+d)\beta+2sd\gamma}{sd} } \varepsilon^{-2}
\end{align*}
to get an error bound of size $\varepsilon$ for some  $\beta >0$ and $h_{\varepsilon}\le h_{\beta}$.
For the third error contribution we have with \eqref{condspectralgapfinal1} the sufficient condition
\begin{align*}
\kllevel_{\varepsilon}^{\frac{1}{2}} h_{\varepsilon}^{-d} \exp\left(- \nsamples_{\varepsilon}  \rho_1 H(\kllevel_{\varepsilon}) \right) \stackrel{!}{\le} \varepsilon \Leftrightarrow \left(\kllevel_{\varepsilon}^{\frac{1}{2}}\varepsilon^{-1}\exp\left(- \nsamples_{\varepsilon}  \rho_1 H(\kllevel_{\varepsilon}) \right)\right)^{\frac{1}{d}} \le h_{\varepsilon}
\end{align*}
to ensure an error bound of size $\varepsilon$. 
We now define $\tilde{\nsamples}_{\varepsilon} \in \N$ by 
\begin{align}\label{tilden}
	\tilde{\nsamples}_{\varepsilon}:=\arg\min\left\{M \in \N :  \kllevel_{\varepsilon}^{\frac{1}{2}}\varepsilon^{-1}\exp\left(- M  \rho_1 H(\kllevel_{\varepsilon}) \right) \le M^{-\frac{1}{2\alpha+1}}\right\}.
\end{align}
In total, in the case $\dimVh > \nsamples^{\frac{1}{2\alpha+1}}$, $\nsamples^{-\frac{2\alpha}{2\alpha+1}} \le d \log(h^{-1})\nsamples^{-1}$, this gives the sufficient conditions
\begin{align}
\label{finalLeps2}
L_{\varepsilon} &= \lceil \varepsilon^{-\frac{2d}{4s+d}} \rceil,\\
\label{finalMeps2}
\nsamples_{\varepsilon} &\gtrsim \max\left\{\tilde{\nsamples}_{\varepsilon},\lceil \varepsilon^{-\frac{2d}{4s+d}} \rceil^{\frac{2(2s+d)\beta+2sd\gamma}{sd} } \varepsilon^{-2} \right\},\\
\label{finalheps2}
 \left(\kllevel_{\varepsilon}^{\frac{1}{2}}\varepsilon^{-1}\exp\left(- \nsamples_{\varepsilon}  \rho_1 H(\kllevel_{\varepsilon}) \right)\right)^{\frac{1}{d}}&\lesssim  h_{\varepsilon} \lesssim  \min\left\{H^{\frac{1}{4s}}(\kllevel_{\varepsilon}) \lambda^{\frac{1}{2s}}_{\kllevel_{\varepsilon}+1},\lambda^{\frac{1}{s}}_{\kllevel_{\varepsilon}}, \nsamples_{\varepsilon}^{-\frac{1}{d(2\alpha+1)}} , h_0
\right\},
\end{align}
to ensure a total error bound of size $3\varepsilon$, where $h_0$ comes from 
\cref{prop:babuska} and \eqref{tilden} ensures that the interval for $h_{\varepsilon}$ is non-trivial.
 
Finally, we consider the case $\dimVh > \nsamples_{\varepsilon}^{\frac{1}{2\alpha+1}}$ and $\nsamples_{\varepsilon}^{-\frac{2\alpha}{2\alpha+1}} \ge d \log(h_{\varepsilon}^{-1})\nsamples_{\varepsilon}^{-1}$. Since $1> \frac{2\alpha}{2\alpha+1}$, we observe that for fixed $h_{\varepsilon}$ this inequality can be satisfied if $\nsamples_{\varepsilon}$ is large enough. 
Furthermore, we observe 
\begin{align*}
\nsamples^{-\frac{2\alpha}{2\alpha+1}} \ge d \log(h^{-1})\nsamples^{-1} &\Leftrightarrow \nsamples^{\frac{1}{2\alpha+1}}  \ge d \log(h^{-1})\\& \Leftrightarrow -d^{-1}\nsamples^{\frac{1}{2\alpha+1}}\le \log(h) \Leftrightarrow \exp\left(-d^{-1}\nsamples^{\frac{1}{2\alpha+1}} \right) \le h.
\end{align*}
This gives a lower bound on $h_{\varepsilon}$ since
\begin{align}\label{thirdchainfirstcase}
 \exp\left(-d^{-1}\nsamples_{\varepsilon}^{\frac{1}{2\alpha+1}} \right) \lesssim  h_{\varepsilon} \lesssim  \min\left\{H^{\frac{1}{4s}}(\kllevel_{\varepsilon}) \lambda^{\frac{1}{2s}}_{\kllevel_{\varepsilon}+1},\lambda^{\frac{1}{s}}_{\kllevel_{\varepsilon}}, \nsamples_{\varepsilon}^{-\frac{1}{d(2\alpha+1)}}  \right\}.
\end{align}
Now we define $\hat{\nsamples}_{\varepsilon} \in \N$ by
\begin{align}\label{hatn}
	\hat{\nsamples}_{\varepsilon}:=\arg\min\left\{M \in \N : \exp\left(-d^{-1}M^{\frac{1}{2\alpha+1}} \right)  \le M^{-\frac{1}{d(2\alpha+1)}}  \right\},
\end{align}
and thus the condition \eqref{thirdchainfirstcase} can be satisfied for $\nsamples_{\varepsilon}\ge \hat{\nsamples}_{\varepsilon}$.
Since we now have $\tilde{\rho}_{h}(\nsamples)=\nsamples^{-\frac{2\alpha}{2\alpha+1}} +d \log(h^{-1})\nsamples^{-1} \le 2 \nsamples^{-\frac{2\alpha}{2\alpha+1}} $, 
we get the chain of inequalities
$(\kllevel_{\varepsilon}^{\frac{1}{2}}+G(\kllevel_{\varepsilon}))\tilde{\rho}^{\frac{1}{2}}_{h_{\varepsilon}}(\nsamples_{\varepsilon})\lesssim  \kllevel_{\varepsilon}^{\gamma} \nsamples_{\varepsilon}^{-\frac{\alpha}{2\alpha+1}}   \lesssim  \kllevel_{\varepsilon}^{\gamma}  \nsamples_{\varepsilon}^{-\frac{\alpha}{2\alpha+1}}$.
This gives us the sufficient condition
 $\kllevel_{\varepsilon}^{\gamma}  \nsamples_{\varepsilon}^{-\frac{\alpha}{2\alpha+1}}\stackrel{!}{\le } \varepsilon \Leftrightarrow \nsamples_{\varepsilon} \ge \kllevel_{\varepsilon}^{\frac{\gamma(2\alpha+1)}{\alpha}}\varepsilon^{-\frac{2\alpha+1}{\alpha}}$
to ensure an error bound of size $\varepsilon$.
For the third error contribution, we derive the inequality
\begin{align}\label{suff3}
 \kllevel_{\varepsilon}^{\frac{1}{2}} h_{\varepsilon}^{-d} \exp\left(- \nsamples_{\varepsilon}  \rho_1 H(\kllevel_{\varepsilon}) \right) \stackrel{!}{\le}\varepsilon \Leftrightarrow  \left(\kllevel_{\varepsilon}^{\frac{1}{2}}\varepsilon^{-1}\exp\left(- \nsamples_{\varepsilon}  \rho_1 H(\kllevel_{\varepsilon}) \right)\right)^{\frac{1}{d}} \le h_{\varepsilon}.
\end{align}
In addition, we can now define 
 $\nsamples^{\prime}_{\varepsilon} \in \N$ by
\begin{align}\label{primen}
	\nsamples^{\prime}_{\varepsilon}:=\arg\min\left\{M \in \N :   \left(\kllevel_{\varepsilon}^{\frac{1}{2}}\varepsilon^{-1}\exp\left(- \nsamples_{\varepsilon}  \rho_1 H(\kllevel_{\varepsilon}) \right)\right)^{\frac{1}{d}}\le  \exp\left(-d^{-1}\nsamples_{\varepsilon}^{\frac{1}{2\alpha+1}} \right) \right\}.
\end{align}
This finally gives the sufficient conditions
\begin{align}
\label{finalLeps3}
&L_{\varepsilon} = \lceil \varepsilon^{-\frac{2d}{4s+d}} \rceil,\\
\label{finalMeps3}
& \nsamples_{\varepsilon} \gtrsim \max\left\{\hat{\nsamples}_{\varepsilon},\nsamples^{\prime}_{\varepsilon},  \kllevel_{\varepsilon}^{\frac{\gamma(2\alpha+1)}{\alpha}}\varepsilon^{-\frac{2\alpha+1}{\alpha}},\left( \min\left\{H^{\frac{1}{4s}}(\kllevel_{\varepsilon}) \lambda^{\frac{1}{2s}}_{\kllevel_{\varepsilon}+1},\lambda^{\frac{1}{s}}_{\kllevel_{\varepsilon}} \right\} \right)^{-d(2\alpha+1)}\right\},\\
\label{finalheps3}
 &\left(\kllevel_{\varepsilon}^{\frac{1}{2}}\varepsilon^{-1}\exp\left(- \nsamples_{\varepsilon}  \rho_1 H(\kllevel_{\varepsilon}) \right)\right)^{\frac{1}{d}}\lesssim  h_{\varepsilon}  \\&\lesssim  \min\left\{H^{\frac{1}{4s}}(\kllevel_{\varepsilon}) \lambda^{\frac{1}{2s}}_{\kllevel_{\varepsilon}+1},\lambda^{\frac{1}{s}}_{\kllevel_{\varepsilon}}, \nsamples_{\varepsilon}^{-\frac{1}{d(2\alpha+1)}} , \exp\left(-d^{-1}\nsamples^{\frac{1}{2\alpha+1}} \right),h_0 \right\}\nonumber
\end{align}
for the case $\dimVh > \nsamples^{\frac{1}{2\alpha+1}}$, $\nsamples^{-\frac{2\alpha}{2\alpha+1}} \ge d \log(h^{-1})\nsamples^{-1}$, which ensure a total error bound of size $3\varepsilon$. Note that we always fix $L_{\varepsilon} $ first, and that we can always satisfy the conditions \eqref{finalMeps},  \eqref{finalMeps2} and \eqref{finalheps3} by choosing $\nsamples_{\varepsilon}$ large enough. 
This then fixes $h_{\varepsilon}$ in \eqref{finalheps} and fixes the potential intervals for $h_{\varepsilon}$ in \eqref{finalheps2} and \eqref{finalheps3}.
Of course, in practice we would always choose  $\nsamples_{\varepsilon}$ as small as possible and $h_{\varepsilon}$ as large as possible.
\end{proof}
Note at this point that a practical goal is to make $h_{\varepsilon}$ as large as possible and $\nsamples_{\varepsilon}$ as small as possible and still ensure an error level of size $3\varepsilon$.
Therefore, we advocate sticking to the second case, where a reasonable choice is possible as in  \eqref{finalheps2} and furthermore for $\nsamples_{\varepsilon}$ as in  \eqref{finalMeps2}. This naturally leads to the tapering estimator.

\subsection*{Parameter choices for Brownian motion}
The univariate Brownian motion (i.e. $d=1$) was given in \eqref{univariateBrownianMotion}, i.e. $R^{(1)}_{B}(x,x^{\prime}):=\min\{x,x^{\prime}\}$
with eigenvalues $\lambda_{R^{(1)}_{B}}(\kllevelrun) \sim \kllevelrun^{-2}$ and the constant from \eqref{defgamma} as $\gamma_{R^{(1)}_{B}}=\frac{3}{2}$. The spatial smoothness is $s=\frac{1}{2}-\delta$ for arbitrarily small $\delta>0$. First we observe for \eqref{finalLeps2} that
$
L_{\varepsilon} = \lceil \varepsilon^{-\frac{2d}{4s+d}} \rceil \sim \lceil \varepsilon^{-\frac{2}{3}} \rceil 
$
From the choice for $h_{\varepsilon}$ we have
$\nsamples_{\varepsilon} \gtrsim\max\left\{\tilde{\nsamples}_{\varepsilon},\lceil \varepsilon^{-\frac{2}{3}} \rceil^{8\beta+3} \varepsilon^{-2} \right\}$.
For the term $\tilde{\nsamples}_{\varepsilon}$, defined in \eqref{tilden},
we first note that solutions $x^{\star}$ of the equation
 $\kllevel_{\varepsilon}^{\frac{1}{2}}\varepsilon^{-1}\exp\left(- M  \rho_1 H(\kllevel_{\varepsilon}) \right) = M^{-\frac{1}{2\alpha+1}}$
can be written in terms of the \emph{product logarithm function} $W$. We have 
\begin{align*}
	x^{\star}= \frac{1}{(2\alpha+1)\rho_1 H(\kllevel_{\varepsilon})} W \left( -(2\alpha+1)\rho_1 H(\kllevel_{\varepsilon}) \left( \kllevel_{\varepsilon}^{-\frac{1}{2}}\varepsilon \right)^{(2\alpha+1)} \right)
\end{align*}
and due to monotonicity we have that $\tilde{\nsamples}_{\varepsilon}=\lceil x^{\star}\rceil$. Moreover, we observe that we can make the bound on $\nsamples_{\varepsilon}$ a bit larger by considering $\alpha=0$, i.e. we can make the bound independent of $\alpha$.

To get a bound on $h_{\varepsilon}$ from \eqref{finalheps2}, we consider the upper bound (since we want to make $h_{\varepsilon}$ as large as possible).
So we see that the term $ H^{\frac{1}{4s}}(\kllevel_{\varepsilon}) \lambda^{\frac{1}{2s}}_{\kllevel_{\varepsilon}+1} $ dominates the eigenvalue term $\lambda^{\frac{1}{s}}_{\kllevel_{\varepsilon}}$ and we get
$h_{\varepsilon}:=\min\left\{ \lceil \varepsilon^{\frac{10}{3}} \rceil , \nsamples_{\varepsilon}^{-1},h_{0} \right\}$.

The case of  higher dimensional Brownian motion could be done in a similar way. Note that $G$ is completely independent of the dimension, while $H$ decays exponentially with the dimension. The latter exponentially affects the choice for $h_{\epsilon}$  in \eqref{finalheps2} and for $M_{\epsilon}$ in \eqref{finalMeps2} for growing dimensions.

\end{document}